\documentclass[12pt,reqno,a4paper]{amsart}
\usepackage[utf8]{inputenc}
\usepackage[T1]{fontenc}

\usepackage{verbatim,footmisc}

\usepackage[mathscr]{eucal}

\DefineFNsymbols*{a}{symbols}
\setfnsymbol{a}

\usepackage{amsmath}
\usepackage{amssymb}
\usepackage{amsthm}
\usepackage{lmodern}
\usepackage{latexsym}
\usepackage{amsfonts}
\usepackage{mathrsfs}
\usepackage[all]{xy}
\usepackage{bm}
\usepackage{yfonts}

\usepackage{fge}
\usepackage{hyperref}
\hypersetup{pdftoolbar=true, pdftitle={Proj_Ellip}, pdffitwindow=true, colorlinks=true, citecolor=blue, filecolor=black, linkcolor=blue, urlcolor=blue, hypertexnames=false}

\numberwithin{equation}{section}

\usepackage[top=4cm, bottom=3.4cm, left=3cm , right=3cm]{geometry}

\DeclareFontFamily{U}{mathx}{}
\DeclareFontShape{U}{mathx}{m}{n}{ <-> mathx10 }{}
\DeclareSymbolFont{mathx}{U}{mathx}{m}{n}
\DeclareFontSubstitution{U}{mathx}{m}{n}

\DeclareMathAccent{\widecheck}{0}{mathx}{"71}

\usepackage{enumerate}

\theoremstyle{plain}
\newtheorem{thm}{Theorem}[section]
\newtheorem{lemma}[thm]{Lemma}
\newtheorem{coro}[thm]{Corollary}
\newtheorem{prop}[thm]{Proposition}

\newtheorem{rem}[thm]{Remark}

\newtheorem{con}[thm]{Conjecture}
\newtheorem*{noti}{Notations}
\newcommand{\bee}{\begin{eqnarray}}
\newcommand{\ene}{\end{eqnarray}}
\newcommand{\bea}{\begin{eqnarray*}}
\newcommand{\ena}{\end{eqnarray*}}
\newcommand{\ud}{\mathrm{d}}

\newcommand*{\X}{\ensuremath{\mathcal{X}}}

\newcommand*{\Y}{\ensuremath{\mathcal{Y}}}

\newcommand*{\GL}{\ensuremath{{\mathrm{GL}}}}
\newcommand*{\C}{\ensuremath{\mathbb{C}}}
\newcommand*{\R}{\ensuremath{\mathbb{R}}}
\newcommand*{\Z}{\ensuremath{\mathbb{Z}}}
\newcommand*{\Q}{\ensuremath{ \mathbb{Q}}}

\newcommand{\fone}[1]{\ensuremath{\frac{1}{#1}}}
\newcommand{\lf}{\left}
\newcommand{\ri}{\right}
\newcommand{\f}{\frac}

\newcommand{\smu}{\hskip 0.5em \sideset{_{}^{}}{^{\ast}_{}}\sum\limits}

\newcommand{\e}[1]{\ensuremath{e\left( #1\right)  }}

\title{First moments of ${\GL}(3)\times {\GL}(2)$ and $ {\GL}(2)$ $L$-functions and their applications}

\subjclass[2010]{11F67, 11F66, 11L05}

\keywords{Automorphic forms, First moment, Rankin-Selberg $L$-functions}

\author[F. Hou]{Fei Hou}
\address{School of Sciences, Xi'an University of Technology,
Xi'an 710054, China} \email{fhou@xaut.edu.cn}

\begin{document}

\maketitle

\begin{abstract}\small{Let $F$ be a self-dual Hecke-Maa\ss\ form for $\GL(3)$ underlying the symmetric square lift of a $\GL(2)$-newform of square-free level and trivial nebentypus. In this paper, we are interested in the first moments of the central values of $\rm GL(3)\times GL(2)$ $L$-functions and $\rm  GL(2)$ $L$-functions. As a result, we obtain an estimate for the first moment for $L(1/2, F\otimes f)$ in a family, where $F$ is of the level $q^2$, and $f\in \mathcal{B}^\ast_k(M)$ for any primes $q,M\ge 2$ such that $(q,M)=1$. We prove the subconvex bound for $L(1/2, F\otimes f)$ involving the levels aspects simultaneously in the range $M^{13/64+\varepsilon }\le q
\le M^{11/40-\varepsilon}$ and $M> q^\delta$ for any $\varepsilon, \delta>0$ for the first time. Moreover, we further investigate the first moments of these $L$-functions in the weight $k$ aspect over $K\le k\le 2K$, with $K$ being a large number. As the results, we obtain a Lindelöf average bound for the first moment of $L(1/2, f)L(1/2, F\otimes f)$ of degree 8 and an asymptotic formula for the first moment of $L(1/2, F\otimes f)$ with an error term of $O(K^{-1/4+\varepsilon})$, respectively.}
\end{abstract}

\section{Introduction}\label{secIntrod}

\setcounter{equation}{0}
\medskip

The subconvexity problem for $L$-functions is a very basic topic in number theory, which plays an important r\^{o}le in arithmetic. In history, for $L$-functions of degree one or two, this problem is completely solved. There are considerable literature on subconvexity of these $L$-functions in various aspects. For degree of three or higher degrees, this becomes very  challenging, and only a limited number of results were known.  

Typically, a powerful way in studying subconvexity is to appeal to the moment method. See notably Ivi\'{c} \cite{Iv}, Conrey-Iwaniec \cite{CI}, Xiaoqing Li \cite{LX1} and Blomer \cite{Bl13} for relevant descriptions and the references therein. Let $N_1, N_2\ge 1$ be two integers, and $k$ an even integer. Let $F(z)$ be a normalized Hecke-Maa\ss\ form of type $\nu=(\nu_1,\nu_2)$ for the congruent subgroup $\Gamma_0(N_1)$ with trivial nebentypus, and $f\in \mathcal{B}^\ast_k(N_2)$ be a Heke newform on $\text{GL}(2)$ of weight $k$ and level $q$ with trivial nebentypus; see \S2.2 for definitions and backgrounds. In this paper, we will be concerned about the central values of $\text{GL}(3)\times \text{GL}(2)$ Rankin-Selberg $L$-function $L(s,F\otimes f)$, and the product of $\text{GL}(3)\times \text{GL}(2)$ and $\text{GL}(2)$ $L$-functions $L(s,F\otimes f)L(s,f)$. This is a very crucial step in understanding automorphic $L$-functions of higher degrees of six and eight.

 In 2012, Xiaoqing Li \cite{LX1} made the first breakthrough to break the convex barriers of $\text{GL}(3)$ $L$-functions in the $t$-aspect as well as $\text{GL}(3)\times \text{GL}(2)$ $L$-functions in the spectral aspect by investigating the first moment of $\text{GL}(3)\times \text{GL}(2)$ $L$-functions. Not long after this, Blomer \cite{Bl13} achieved the subconvexity for $\text{GL}(3)\times \text{GL}(2) \times \text{GL}(1)$ $L$-functions and $\text{GL}(3) \times \text{GL}(1)$ $L$-functions in their conductors aspects. With many sophisticated methods being developed in recent few years such as $\delta$-method (see, for example, \cite{M1,Mu2,Mu1}), amplification method (see, for example, \cite{HN}), deep analysis of the average involving the Kloosterman sums (see, for example, \cite{KMS,Shp}) as well as the integral transforms (see, for example,  \cite{Mu2}) to name a few, the researchers were gradually attempting to solve the following well-known subconvexity conjecture: 
\begin{con} Let $N_1, N_2\ge 1$ be two integers. Let $k$ be an even integer and $t\ge 1$. For any normalized ${\GL}(3)$-Maa\ss\ form $F$ of level $N_1$ and ${\GL}(2)$-Maa\ss\ form $f$ of level $N_2$ and weight $k$, there holds that $L(1/2+it, F\otimes f )\ll_\varepsilon (N_1N^{3/2}_2k^{3}t^3)^{1/2-\varepsilon}$ for some absolute constant $\varepsilon>0$.
\end{con} 

 We will collate some results on this conjecture in different regards. 
\begin{itemize}
\item In history, Khan \cite{Kh1} firstly broke the convex bound in the $N_2$-aspect. In the case of $L(1/2, F\otimes f)\ge 0$, he exploited the amplification method to establish a conditional subconvexity \[L(1/2, F\otimes f )\ll N^{3/4-1/2001+\varepsilon}_2,\]  when $N_2$ is a prime (which should be though to be the hardest case). 
\item Sharma and Sawin \cite{Sh} employed a type of `Mass transfer' as given by Munshi \cite{M27}, showing that
\[L(1/2, F\otimes f)\ll p^{3/2-1/16+\varepsilon}\]
 for any $f\in \mathcal{B}_k(p^2, \chi^2)$ with the nebentypus $\chi^2$ being of modulus $p^2$ (where $p$ is a prime). Later, Kumar, Mallesham and Singh \cite{KMS} gave a better exponent of $3/2-3/20+\varepsilon$ in the depth aspect, where $f\in \mathcal{B}_k(p^{2r}, \chi^{2})$ with the nebentypus $\chi^2$ being of modulus $p^r$, $r>3$. 
\item Concerning the $t$-aspect subconvexity for $\text{GL}(3)\times \text{GL}(2)$, Munshi \cite{Mu1} developed the $\delta$-method to its furthest achievement by successfully establishing the following  \[L(1/2+it, F\otimes f )\ll t^{3/2-1/42+\varepsilon}.\] Subsequently, there were in droves many works devoted to the strengthened (or the generalized) subconvexity; see, for example, \cite{LS, Huang, HKS}. The best exponent should be $3/2-3/20+\varepsilon$ due to Lin and Sun \cite{LS}.
\item Recently, Kumar \cite{Ku} firstly proved the subconvex bound in the weight $k$ aspect. By using the `conductor lowering mechanism' developed by Munshi \cite{Mu2}, he showed that 
\[L(1/2+it, F\otimes f )\ll k^{3/2-1/51+\varepsilon}.\]
\item  Regarding the level aspect  subconvexity involving the $
\text{GL}(3)$ forms, it is completly out of reach. By currently developed techniques, it appears that one cannot get any power-saving to beat the subconvex barrier whatsoever. More recently, Kumar, Munshi and Singh \cite{KRS2} were able to prove that 
\bee \label{0oo9fif}L(1/2, F\otimes f ) \ll  (N_1N^{3/2}_2)^{1/2+\varepsilon}\lf ( \f{N^{1/4}_1}{N^{3/8}_2} +\f{N^{1/8}_2}{N^{1/4}_1} \ri),\ene when $N_1,N_2$ are both co-prime primes. This hybrid bound is subconvex in the range $N^{1/2+\varepsilon}_2<N_1<N^{3/2-\varepsilon}_2$.
\end{itemize}

\subsection{Main results}In this paper, we are more concerned about the self-dual $\GL(3) \times  \GL(2)$ $L$-functions, which have important implications to equidistribution problems of $L^q$-mass of cusp forms (as $q\rightarrow \infty$) and Heegner points on Shimura curves as well as the Quantum Unique Ergodicity (QUE) conjecture (see, for example, \cite{BH,HT,B,LMY,L,HM,BKY,PN1}). For instance, the well-known Watson's formula implies that \[\| g\|_4^4\,\approx \fone{p^2}\sum_{h \in\mathcal{B}_{2k}^\ast(p) }L(1/2,h )L(1/2,\text{sym}^2 g \otimes h )\] for any $\GL(2)$-newform $g\in \mathcal{B}^\ast_k(p)$ of large fixed weight $k$ and prime level $p$. Notice that the level aspect subconvexity for self-dual $\GL(3)$ forms is still wide-open for specialists, albeit the notable successes of subconvexity for $\GL(3)$ $L$-functions in various aspects over the past decade. One of main objectives of the present paper is to address the subconvexity in the levels aspects for $\GL(3)$ and  $\GL(2)$ forms simultaneously\footnote[1]{One might wander whether the method of Kumar, Munshi and Singh can be adpated to obtain subconvexity for $L(1/2, \text{sym}^2g\otimes f )$ with $F=\text{sym}^2g$ being a symmetric square lift of a $\GL(2)$-newform $g$, say. Unfortunately, this seems impracticable. The underlying reason is that the modulus $\mathcal{Q}$ in equipping with $\delta$-symbol method is usually large relative to the level of $F$; see \cite[Section 2-2B]{KRS2}. This gives rise to a deadlock since of an inevitable
gap for $\GL(3)$-Vorono\u{\i} formulae which we will explain after a while.}. To be specific, we prove the following
results:
\begin{thm}\label{3243251515} Let $q,M\ge 2$ be two primes satisfying that $(q,M)=1$ and $M^{1/8}<q<M^{3/8}$. Let $F$ be a Hecke-Maa\ss\ form for the congruence subgroup $\Gamma_0(q^2)$ obtained as the symmetric square lift of a $\GL(2)$-newform of level $q$ with trivial nebentypus. Define the set \bee \label{x342400s}\mathscr{P}_L=\{p: p \text{ is a prime },(p,qM)=1,L\le p\le 2L\}\ene with  
\[\f{q^{1+\varepsilon}} {{M^{1/4}} }+ \f{{M} ^{1/4}}{{q^{1-\varepsilon}} }   
\le L\le M^{1/8-\varepsilon}
.\]Then, for any $\ell =\{1,\ell_1\ell_2, \ell^2_1\ell^2_2\}$ with $\ell_1,\ell_2\in \mathscr{P}_L$, we have 
 \begin{align}\label{00vo9493x4t422} &\sum_{f\in \mathcal{B}^\ast_k(M)}\omega^{-1}_f \lambda_f(\ell)L(1/2, F\otimes f)\ll (qM)^\varepsilon 
\left ( \f{1}{\sqrt{\ell}} 
+q\sqrt{\f{\ell }{ M}}+\f{q^2\ell}{M^{5/8}}
 +\f{ q\ell^{2} }{M^{5/8}} +\f{q^{7/2} \ell^{7/4} }{M^{9/8}} \right ),
\end{align}where the implied constant depends only on the weight $k$ and the Langlands parameters $\alpha_i$, $1\le i \le 3$, of the form $F$ as given in \eqref{de040400444}.
\end{thm}
Observe that under the self-duality for $\rm GL(3)$ forms, the non-negativity for the central values of the $\GL(3) \times  \GL(2)$ $L$-functions is guaranteed by Lapid's method \cite{Lap}; see also \cite[Remark 3]{LX1}. As an application of Theorem \ref{3243251515}, we are thus able to obtain the following subconvex bound:
 \begin{coro}\label{corodoe923} Let the forms $F,f$ and the parameters $q,M,k$ be as in Theorem \ref{3243251515}. Let $\X\ge 1$ such that $ (M^{3/2}q^2)^{1-\delta}\le \X\le  (M^{3/2}q^2)^{1+\varepsilon}$ for arbitrarily small constants $\varepsilon,\delta>0$. Then, for any $L$ with
\bee \label{c9vv0343}\f{q^{1+\varepsilon}} {{M^{1/4}} }+\f{M^{(1+3\delta)/4+\varepsilon}}{q^{1-\delta}} \le L\le \min {\lf \{ M^{3/64-\varepsilon}, \f{M^{3/32-\varepsilon}}{q^{1/4}},   \,  \f{{M}^{1/8}}{q^{5/14+\varepsilon}}, \, \f{M^{7/40-\varepsilon}}{q^{3/10}}\ri\}},\ene
 we have 
\begin{align}\label{426542vc} &L(1/2, F\otimes f) \ll \X^{1/2+\varepsilon +\delta}
\left (  \f{M^{1/4}}{qL} +\f{L^2}{M^{1/4}}+\f{L^4q}{M^{3/8}}+
 \f{L^8}{M^{3/8}}
+ \f{ L^{7} q^{5/2}}{M^{7/8}} +\fone{\sqrt{M}} \right ).
\end{align}
\end{coro}
This establishes the subconvexity for any $q,M $ with $M^{13/64+\varepsilon }\le  q\le M^{11/40-\varepsilon}$ and $M> q^\delta$ for any $\varepsilon, \delta>0$, and should be the first instance for the levels aspects subconvexity of self-dual $\GL(3)\times \GL(2)$ $L$-functions. Taking a little effort, it is not hard to attain the subconvexity bounds for any $q,M$ varying in this range. For instance, when $q\asymp M^{1/4}$, we have
\[L(1/2, F\otimes f) \ll \X^{1/2-1/80+\varepsilon}  \]
and when $q\asymp M^{11/40-\varepsilon}$, we obtain
\[L(1/2, F\otimes f) \ll \X^{1/2-3/82+\varepsilon} . \]

In this paper, another motivation is to investigate the first moments of $L(1/2, F\otimes f)$ and $L(1/2, f)L(1/2, F\otimes f)$ in the weight aspect in families. We continue to explore the first moment method by adding a `harmonic’ average of the weight $k$, with $K< k\le 2K$. As a result, we are able to establish the following Lindelöf average bound:\begin{thm}\label{32432515777}Let $F(z)$ be a normalized self-dual Hecke-Maa\ss\ form on ${\GL}(3)$. Let $K$ be a sufficiently large parameter. Let $W$ be a smooth function compactly supported on $[1/2,5/2]$ with bounded derivatives. Then, we have 
 \begin{align}\label{009koieu3}  &    \sum_{k \equiv 0 \bmod 2}W{\lf ( \f{k-1}{K}\ri)}   \sum_{f\in \mathcal{B}^\ast_k}\omega^{-1}_fL(1/2,f)L(1/2, F\otimes f)\ll K^{1+\varepsilon},
 \end{align}
 where the implied constant depends only on the constant $\varepsilon$, the form $F$ as well as the level of the form $f$.
\end{thm}
In addition, we are able to obtain the following asymptotic formula for the $\GL(3) \times  \GL(2)$ $L$-functions in the weight aspect with a power-saving error term: 
\begin{thm}\label{c45254544d}Keep the notation as in Theorem \ref{32432515777}. Then, we have
\begin{align}\label{009koieu3ss}  \sum_{k \equiv 0 \bmod 2}W{\lf ( \f{k-1}{K}\ri)}   \sum_{f\in \mathcal{B}^\ast_k}\omega^{-1}_f
L(1/2, F\otimes f) =\f{L(1, F)K}{4}\widehat{W}(0)+O(K^{-1/4+\varepsilon}),
 \end{align}
 where $\widehat{W}$ denotes the ordinary Fourier transform of $W$.
\end{thm}
Consequently, as a by-product we can obtain a cute result for the non-vanishing problem of $L(1/2, F\otimes f)$ as follows:
\begin{coro}\label{32432515779297}For any self-dual Maa\ss\ from $F$ on $\GL(3) $, there exists a newform $f\in \mathcal{B}^\ast_k$ with $k$ sufficiently large, such that $L(1/2, F\otimes f)\neq 0$.
\end{coro}
It is instructive to see that the estimate in \eqref{009koieu3} should be the best-possible in essence, and is already able to
beat what the classic spectral large sieve (see, for example, \cite[Theorem 2]{DI}) together with the Cauchy-Schwarz
inequality implies, which only provides a quannity of $O(K^{2+\varepsilon})$. This also reflects the significant effectiveness by further averaging over the parameter $k\sim K$, when it comes to investigating the subconvexity problem involving the weight aspect. We also refer the readers to the works of Luo \cite{Luo12} and Khan \cite{Kh13} for relevant heuristics. It, on the other hand,  should be noticed that, thanks to the rational structure of the exponential factor in \eqref{x3r3209c323114}, we would succeed to achieve this strong Lindelöf average bound; see Remark \ref{cf34334t48033} below for details. Besides, it is also believed that our method in proving Theorem \ref{009koieu3} can be naturally adapted to produce simultaneous non-vanishing of $\GL(3) \times \GL(2)$ and $\GL(2)$ $L$-functions at the central point, and even subconvexity for $L(1/2, F\otimes f)$ or $L(1/2,f)L(1/2, F\otimes f)$ in the weight and the level (of $F$ or $f$) aspects. We are now working on this direction, and expect to report the progress in the future article.

The initial manoeuvre of our proofs of
main results follows the path pioneered by Xiaoqing Li and Blomer, who looked into the average of first moment of $\rm GL(3) \times GL(2)$ $L$-function in a family. The strategies we pursued in the paper, however, can be understood as the ``Kuznetsov-Vorono\u{\i}-Reciprocity-Vorono\u{\i}-Bilinear form'' in proving Theorem \ref{3243251515} and ``Kuznetsov-Exponential sums/Vorono\u{\i}'' in proving Theorems \ref{32432515777} \& \ref{009koieu3ss}, respectively. In particular, for the former, this is reminiscent of an important work of Blomer and Khan, \cite{BK19}, in which they investigated the twisted moments of $\text{GL}(3)\times \text{GL}(2)$ and $\text{GL}(2)$ automorphic $L$-functions, and established the following reciprocity formula
\begin{align} \label{crtty42342}\fone{\sqrt{q}}\sum_{f \text{ level } q} L(1/2, F\otimes  f) L(1/2,f) \lambda_f(\ell) \, \leftrightsquigarrow \, \fone{\sqrt{\ell }}\sum_{f \text { level } \ell} L(1/2, F\otimes f) L(1/2,f) \lambda_f(q)   \end{align}
by the ``Kuznetsov-Vorono\u{\i}-Reciprocity-Vorono\u{\i}-Kuznetsov'' law. Among other things, one salient point is the application of the reciprocity law, which, as the Kloosterman sum of dimension one, keeps us going back to somewhere. However, when applying the Vorono\u{\i} to the Kloosterman sum $S(n\ell, 1;c)$ of dimension two, the inverse of the variable $\ell$, that is, the exponential factor like $S(n\overline{\ell},\sigma;c)$ for some less significant parameter $\sigma$, will be occurred; this will increase the dimensions of Kloosterman sums afterwards, if one were to resort to the Vorono\u{\i} again. In this sense, we can barely expect the reciprocity formula for Hecke-Maa\ss\ forms of higher degrees $d\ge 4$ in the levels aspects\footnote[2]{Alternatively, one might explore the reciprocity by equipping with the Kuznetsov formulae in the level aspect for ${\rm{GL}}(n)$, $n\ge 3$; see, for example, \cite{BBM} for descriptions on ${\rm{GL}}(3)$. However, observe that, it seems very tricky to establish the relation $S(m_1m_2\cdots m_k{c_1}\overline{\ell},1; c_2) \leftrightsquigarrow S(n_1n_2\cdots n_k{c_1}\overline{q},1; c_2)$ for the levels $\ell,q\ge1$, where $m_i,n_i$ are related to the summations of Fourier coefficients of $\GL(d_i)$-forms, $d_i\ge1$. It is still not sure that if a kind of reciprocity law involving the $\GL(3)$ level aspect can be achieved.}. It seems that the possible symmetric relation like \eqref{crtty42342} for $\text{GL}(d)\times \text{GL}(2)$ involving the levels aspects is only for the $\text{GL}(3)\times \text{GL}(2)$ case (we refer the interested readers to Jana and Nunes's recent work \cite{JN} for the generalization to $\rm{PGL}_{d+1}\times\rm{PGL}_{d-1} $). In contrast to the $\text{GL}(3)\times \text{GL}(2)$ case, we notice that, for the $\text{GL}(2)\times \text{GL}(1)$ case, Conrey \cite{Con} proved the conductors aspects reciprocity formula
\[\fone{\sqrt{p}}\sum_{\chi \bmod p}|L(1/2, \chi)|^2\chi(\ell) 
\, \leftrightsquigarrow \, \fone{\sqrt{\ell}}\sum_{\chi \bmod \ell}|L(1/2, \chi)|^2\chi(-p) .\] This, however, was further extended by Young \cite{Yo} and Bettin \cite{Be} subsequently.

Let us now elucidate two of main obstacles in the present paper. It is remarkable that, by the current development of the Vorono\u{\i} formulae for $\GL(3)$ in the level aspect, there is still a gap for arbitrary non-zero integer $c\in \Z$  (see Lemma \ref{2017303534623421} below). If the
additive twist {\em{colludes}} with the level $N$ of the form $F$, namely, the parameters $c,N$ satisfy that $(N/(N,c),c)\neq 1$. For instance, one takes $N=q^2$ and $c=c^\prime  q$ with $(c^\prime
, q) = 1$ for any prime $q$. The ${\GL}(3)$-Vorono\u{\i} formula in this case looks rather tricky. In principle, Corbett's formula \cite[Theorem 1.1]{C} can cover this, but it requires a non-trivial analysis of the $p$-adic Bessel transforms, which, however, becomes more involved; see \cite{C} for relevant details. We also refer the readers to the work of Booker, Milinovich and Ng \cite{BMN}, who worked out the Vorono\u{\i} formula in the ${\GL}(2)$-case (if the additive twist {{colludes}} with the level). As an innovation in this paper, we are able to overcome this deadlock, which  makes the hybrid subconvexity for $\GL(3)$ and $\GL(2)$ in the levels aspects come into being. The key point, however, is to invoke the reciprocity law as mentioned above. Indeed, with a suitable constraint for the parameter $\ell$ we are in the unramified case, when applying the Vorono\u{\i} for the first time; then, one would be transferred to the ramified case after an implementation of the reciprocity law. This makes it possible for one to equip with the Vorono\u{\i} again.

The second obstacle is the root number. It is well-known that the explicit determination of root numbers play a focal r\^{o}le when doing averages of central values of $L$-functions in different families. In this paper, we will give a detailed description of the root number $\varepsilon(F\otimes f)$ for any $\GL(3)$ Hecke-Maa\ss\ form $F$ with arbitrary level $N$ in the
framework of the adelic interpretations, which is of independent interest; see \S2.3. Part of descriptions in the special case where $F$ arises as a symmetric square lift from $\GL(2)$ has already been given by Holowinsky, Munshi and Qi \cite{HMQ}. As it can be seen later, the root number $\varepsilon( F\otimes f)$ depends only on the parameter $k$ and the form $F$. One of its factor $i^k$ will have a direct influence on the `harmonic' averages of the Bessel function $J_{k-1}$ in considering the moments of $L$-values in the weight $k$ aspect; see Lemmas \ref{201910000t666011svffs111} and \ref{20191000s111}.

To expose everything as clearly as possible, we have assumed that the parameter $M$ is a prime, which simplifies the argument when applying Petersson trace formula, but our method could in principle be amenable to handling the general case of square-free level
with a bit more effort; see the beginning of \S3.2 for relevant details.\\
\\
{\bf{Acknowledgements.}} The author is very grateful to Profs. Peter Humphries and Guangshi L\"{u} for helpful discussions during the preparation of this work. 
\begin{noti}
 Throughout the paper, $\varepsilon$ (resp. $\delta$) always denotes an arbitrarily small positive constant which may not be the same at each occurrence. $A\ll _\varepsilon B $ means that $|A|\ll X^\varepsilon |B|$, and $n\sim N$ means that $N< n \le 2N$ for every $n\in \Z$. For any positive integers $m,n$, $(m,n)$ denotes the great common divisor of $m$ and $n$. $\mu$ is the M\"{o}bius function.   \end{noti}
\section{Prerequisites}
The purpose of this section is to develop (or supply) certain preliminaries which will be assembled in the bodies of proofs of the main results.
\subsection{Automorphic forms}We will give a recap of some fundamental facts about cuspidal forms (see, for example, Iwaniec and Kowalski's book \cite{IK}). Let $k \ge 2$ be an even integer and $M > 0$ be a square-free integer. Let $\chi$ be a primitive character of modulus $r$ such that $M|r$, which satisfies that $\chi(-1)=(-1)^k$. We denote by $\mathcal{S}_k(M,\chi)$ the vector space of holomorphic cuspidal forms on $\Gamma_0(M)$ with nebentypus $\chi$ and weight $k$. For any $f\in \mathcal{S}_k(M,\chi)$, one then has a Fourier expansion
\[f(z)=\sum_{n\ge 1}\psi_f(n)n^{\f{k-1}{2}} e(nz)\]
for $z\in \mathbb{H}$. Here, $e(z)$ is shorthand for $e^{2\pi i z}$ for any $z\in \C$, and $\mathbb{H}$ means the upper half-plane. The space $\mathcal{S}_k(M,\chi)$ is a finite dimensional Hilbert spaces which can be equipped with the Petersson inner products
\[  \langle   f_1,f_2 \rangle = \int_{\Gamma_0(N)\backslash  \mathbb{H} } f_1(z)\overline{f_2(z)} y^{k-2} \ud x \ud y.\]
We recall the Hecke operators $\{T_n\}$ with $(n,M)=1$, which satisfy the multiplicativity relation
\bee \label{130} T_n T_m=\sum_{d|(n,m)} \chi(d)T_{\f{nm}{d^2}}.\ene Moreover, for any $f_1,f_2\in \mathcal{S}_k(N,\chi)$, one has \[ \langle   T_n f_1,f_2 \rangle=\overline{\chi(n)} \langle   f_1,T_nf_2 \rangle.\]

 One might thus find an orthogonal basis $\mathcal{B}_k(M,\chi)$ of $\mathcal{S}_k(M,\chi)$ consisting of common eigenfunctions of all the Hecke operators $T_n$ with $(n,N) = 1$. For
each $f \in \mathcal{B}_k(M,\chi)$, denote by $\lambda_f(n)$ the $n$-th Hecke eigenvalue,
which satisfies \[T_n f(z)=\lambda_f(n) f(z)\] for all $(n,M)=1$. From \eqref{130}, one has seen that
\[\psi_f(m)\lambda_f(n)=\sum_{d|(n,m)} \chi(d) \psi_f \lf( \f{mn}{d^2} \ri) \]for any $m, n\ge1$ with $(n,M) = 1$. In particular, $\psi_f(1)\lambda_f(n)=\psi(n)$, if $(n,M)=1$. Therefore,  if $(mn,M)=1$,\bee \label{131}\overline{\lambda_f(n)}=\overline{\chi(n)}\lambda_f(n),\quad \lambda _f(m)\lambda_f(n)=\sum_{d|(n,m)}\chi(d) \lambda_f \lf( \f{mn}{d^2} \ri).\ene

The Hecke eigenbasis $\mathcal{B}_k(M,\chi)$ also contains a subset of newforms $\mathcal{B}^\ast_k(M,\chi)$, those forms which are simultaneous eigenfunctions of all the Hecke operators $T_n$ for any $n\ge1$, and normalized to have first Fourier coefficient $\psi_f (1) = 1$. The elements of $\mathcal{B}^\ast_k(M,\chi)$ are usually called primitive forms in the sense of Atkin-Lehner; that is, each one is
orthogonal to all oldforms, and is an eigenfunction of all the Hecke and Atkin-Lehner
operators. For convenience, we abbreviate the notation $\mathcal{B}^\ast_k(M)$ to $\mathcal{B}^\ast_k$, the space of newforms of weight $k$ and arbitrary level parameter, if one does not specify the level parameter.

\subsection{Automorphic $L$-functions and functional equations}  In this part, we shall further shed some light on automorphic $L$-functions. Let $k\ge 2$ be an even integer. We assume that $N\ge 1$ is a positive integer, satisfying that $(N,M)=1$. For any newform $f \in \mathcal{B}_{k}^\ast(M)$, we define the attached $L$-function $L(f,s)$, for $\text{Re}(s)>1$, by the Dirichlet series $\sum_{n\ge 1}\lambda_f(n)n^{-s}$ which can be expressed in terms of an Euler product
\[\prod_{p\nmid M}  \lf (1-\f{\lambda_f(p)}{p^s}+\fone{p^{2s}}\ri)^{-1} \prod_{p\mid M}  \lf ( 1-\f{\lambda_f(p)}{p^s}\ri)^{-1}.\]
 We proceed by defining the completed $L$-function 
\[\Lambda(s,f) =   \mathcal{Q}^{s/2}_f L_\infty (s,f)L(s,f),\]with \[L_\infty (s,f)= \pi^{-s} \Gamma{\lf (\f{s}{2}+\f{k-1}{4} \ri)}  \Gamma{\lf (\f{s}{2}+\f{k-1}{4} \ri)} 
\] and the (analytic) conductor $\mathcal{Q}_f$ being $\mathcal{Q}_f\asymp M$. The completed $L$-function thus admits an analytical continuation to all $s\in \C$, and satisfies that
\[\Lambda(s,f)=\varepsilon(f) \Lambda(1-s,f),\]where the root number $\varepsilon(f) =i^k \mu(M) \lambda_f(M)\sqrt{M}=\pm i^k$.

Let us now put   \bee \label{efkwof334142}\Gamma_0(N)= \lf\{   g\in \text{GL}(3,\Z):g \equiv \left(\begin{matrix} \ast &\ast&\ast \\
\ast &\ast&\ast   \\0&0&\ast \end{matrix}\right)    \mod N  \ri\}.\ene
Let $F(z)$ be a normalized Hecke-Maa\ss\ form of type $\nu=(\nu_1,\nu_2)$ for the congruent subgroup $\Gamma_0(N)$ with trivial nebentypus, which has a Fourier-Whittaker expansion with the Fourier
coefficients $A_F (m, n)$. The Fourier coefficients of $F$ and that of its contragredient $\widetilde{F}$ are
related by $A_F (m, n) = A_{\widetilde{F}} (n, m)$ for any $(mn, N) = 1$, with $A_F (1, 1) = 1$; see, for example, \cite[Chapter 6]{Go} for definition and backgrounds. The Jacquet-Langlands $L$-function is given by $L(s,F)=\sum_{n\ge 1}A_F(n,1)n^{-s}$ for $\text{Re}(s)>1$. Now, let \bee \label{de040400444}\alpha_1=-\nu_1-2\nu_2+1, \quad\alpha_1=-\nu_1+\nu_2, \quad \alpha_1=2\nu_1+\nu_1+1  \ene be the Langlands parameters of $F$. We define the completed $L$-function \[\Lambda(s, F)=\mathcal{Q}^{s/2}_F \varepsilon(F) L_\infty (s,F)L(s,F),\]where 
\[L_\infty(s,F)=\pi^{-3s/2}\prod_{i=1}^3\Gamma{\lf ( \f{s+\alpha_i}{2}\ri)},\] the (analytic) conductor $\mathcal{Q}_F\asymp N$, and $|\varepsilon(F)|=1$. One thus have the following functional equation \[\Lambda(s,F)=\varepsilon(F) \Lambda(1-s,
\widetilde{F}).\] Next, for any $f\in \mathcal{B}^\ast_k(M)$, we turn to considering the $\text{GL}(3)\times\text{GL}(2) $ Rankin-Selberg $L
$-function $L(s,F\otimes f)$, which is given by the Dirichlet series \[L(s,F\otimes f)=\sum_{m,n\ge 1}\f{A_F(n,m)\lambda_f(n)}{(nm^2)^{s}}\] for $\text{Re}(s)>1$. Now, one defines the Gamma factor $L_\infty (s, F\otimes f)$ by \bee \label{3rqr34 r}\begin{split}L_\infty (s, F\otimes f)&=\pi^{-3s} \prod_{i=1}^3\Gamma{\lf ( \f{s+{(k-1)}/{2}+\alpha_i}{2}\ri)} \Gamma{\lf ( \f{s+{(k+1)}/{2}+\alpha_i}{2}\ri)}
\end{split}\ene  and the (analytic) conductor $\mathcal{Q}_{F,f}\asymp N^2M^3$. The completed $L$-function defined by \[\Lambda(s, F\otimes f)=\mathcal{Q}^{s/2}_{F, f} L(s, F\otimes f)\] is thus an entire function with an analytic continuation to all $s\in \mathbb{C}$, and satisfies the function equation that \[\Lambda(s, F\otimes f)=\varepsilon(F\otimes f)\Lambda(1-s, \widetilde{F}\otimes f).\]
\subsection{Epsilon constant}In this part, we will be dedicated to developing some facts on the epsilon constant $\varepsilon(F\otimes f)$, which play a key r\^{o}le in later analysis. We proceed by borrowing some points of view from the automorphic representations of algebraic groups over local fields. Let the parameters $k,N,M$ be as before. Define  \[k_{2,p}= \lf\{   g\in \text{GL}(2,\Z_p):g \equiv \left(\begin{matrix} \ast &\ast \\0&\ast \end{matrix}\right)    \bmod p  \ri\}\] and
  \[k_{3,p}= \lf\{   g\in \text{GL}(3,\Z_p):g \equiv \left(\begin{matrix} \ast &\ast&\ast \\
\ast &\ast&\ast   \\0&0&\ast \end{matrix}\right)    \bmod p \ri\}.\] 
For any given holomorphic newform $ f\in L^2(\Gamma_0 (M) \backslash \text{SL}(2,\R))$ of weight $k$, let us proceed by defining a function $\varphi_f  \in L^2( \text{GL}(2,\Q) \backslash \text{GL}(2,\mathbb{A}_\Q) / \text{K}(2,M))$ by
\[\varphi_f(\gamma g_\infty k)=f(g_\infty)\]for $\gamma\in \text{GL}(2,\Q)$, $g_\infty \in \text{GL}(2,\R)$ and $k\in \text{K}(2,M)$, with \[\text{K}(2,M)= \text{O}(2,\R) \times \prod _{p|M}k_{2,p} \times \prod_{p\nmid  M}\text{GL}(2,\Z_p).\] It can be seen that each function $\varphi_f$ generates the automorphic representation $\pi_f$ of $\text{GL}(2,\mathbb{A}_\Q)$ with trivial central character. We refer the reader to, for example, \cite[Sections 13.3-13.6]{GH2} for relevant backgrounds. 

For any given Hecke-Maa\ss\ form $F$ of level $N$ on $\text{GL}(3)$, we now let $\Pi_{F}$ be the associated automorphic representation, which belongs to the space $L^2(\Gamma_0(N) \backslash {\text{SL}}(3,\R))$. Here, the definition of $\Gamma_0(N) $ can be referred to \eqref{efkwof334142} in \S2.1. The representation $\Pi_F$ is  generated by the smooth functions $\varphi_{\Pi_F} \in  L^2( \textrm{GL}(3,\Q) \backslash \textrm{GL}(3,\mathbb{A}_\Q)  / \text{K}(3,N))$, which are defined in the following way
\[\varphi_{\Pi_F}(\gamma g_\infty k)=F(g_\infty)\]for any $\gamma\in \text{GL}(3,\Q)$, $g_\infty \in \text{GL}(3,\R)$ and $k\in \text{K}(3,N)$, with \[ \text{K}(3,N)=\text{O}(3,\R) \times \prod _{p|N}k_{3,p}\times  \prod_{p\nmid  N}\text{GL}(3,\Z_p).\] It is well-known that one
might factor \[\Pi_F=\Pi_{F,\infty} \otimes_{p< \infty }\Pi_{F,p} ,\]
where $\Pi_{F,p}$ is an irreducible, admissible and unitary representations of $\text{GL}(3,\Q_p)$ for every $p< \infty$.

We shall now ready to deduce an explicit form of the epsilon constant $\varepsilon(\Pi_{F,p}\otimes \pi_{f, p})$ at each place $p\le \infty$. Observe that, for any unramified places $v=p$ with $p\nmid NM$, there already hold that $\varepsilon(\Pi_{F,p}\otimes \pi_{f, p})=1$. We will focus on the archimedean place $v=\infty$ and the finite places $v=p, p|NM$ in the following.
\begin{itemize}
\item  To begin with, at {archimedean place} $p=\infty$, in the Hecke-Maa\ss\ form setting $\Pi_{F, \infty}$ is an unramified principal series representation; while, $\pi_{f, \infty}$ is a ramified discrete series
representation of weight $k$, by the multiplicativity of epsilon factors we have
\[\varepsilon(\Pi_{F,\infty}\otimes \pi_{f, \infty})=\varepsilon( \pi_{f, \infty})^3=i^{k}.\]
\item Now, for $p|M$, the local component $\Pi_{F, p}$ is an unramified principal series representation, and $\pi_{g, p}$ is the
Steinberg representation, one finds
\[\varepsilon(\Pi_{F,p}\otimes \pi_{f, p})=\varepsilon(  \pi_{f,p})^3 =-\lambda_f(p)\sqrt{p}.\] 
\item For $p|N$, one observes that the local component $\Pi_{F, p}$ is the Steinberg representation; whilst, $\pi_{g, p}$ is an unramified principal series representation with trivial central character, so that
\[\varepsilon(\Pi_{F,p}\otimes \pi_{f, p})=\varepsilon(\Pi_{F,p})^2,\] 
with $|\varepsilon(\Pi_{F,p})|=1$.
\end{itemize} We finally obtain
\bee \label{cldodd32f}
\varepsilon( F\otimes f)=\prod_{p \le \infty}\varepsilon(\Pi_{F,p}\otimes \pi_{f, p})=-i^k \lambda_f(M)\sqrt{M}\prod_{p|N}\varepsilon(\Pi_{F,p})^2.
\ene
\subsection{Vorono\u{\i} summation formula}Let $w$ be a compactly supported smooth function on $(0, \infty)$, and $\widetilde{w}$ be its Mellin transform. For any $\rho=0,1$ and any Hecke-Maa\ss\ form $F$ of level $N$ with the Langlands parameters $\alpha_i$, $1\le i\le 3$, being as in \eqref{de040400444}, we define \bee \label{49596v5b5u}\gamma_{\rho}(s)= \varepsilon(F)  i^\rho  \pi^{3(s-1/2)}\prod_{j=1}^3  \f{ \Gamma\lf (  \f{1+s+\alpha_j+\rho }{2}\ri)  }{  \Gamma\lf ( \f{-s-\alpha_j+\rho}{2}\ri) }  \ene and set \[\gamma_\pm (s)=\fone{2} \lf(\gamma_0(s) \mp i \gamma_1(s)\ri).\] Now, we proceed to define two auxiliary quantities 
 \bee \label{00rjfiurjrtt}\Omega_\pm(x;w)=\fone{2 \pi i} \int _{(\sigma)}x^{-s}  \gamma_{\pm}(s) \widetilde{w}(-s)\ud s ,\ene  where $\sigma>-1+\max\{-\text{Re}(\alpha_1), -\text{Re}(\alpha_2),-\text{Re}(\alpha_3) \}$. We then have the following Vorono\u{\i} summation formula for Hecke-Maa\ss\ forms on $\text{GL}(3)$ in the level aspect (see \cite{Zh}).
\begin{lemma}\label{2017303534623421}
Let $w(x)\in 
\mathcal{C}^\infty_c (0,\infty)$ be a smooth function compactly supported on $[1/2,5/2]$ with bounded derivatives. Let $a,\overline{a},c \in \Z $ with $c\neq 0$, $(a,c)=1$, satisfying that $a\overline{a} \equiv 1 \bmod c$. 
\begin{itemize}
\item  If $(cm,N)=1$, we then have
\bee\label{34534vof3534}\begin{split}&\sum_{n>0} A_F(m,n)\e{ \f{n\overline{a}}{c}}w\lf(\f{n}{X}\ri) \\
& \hskip 1.2cm =c\sqrt{N}\sum_{\pm }\sum_{n_1|cm}\sum_{n_2>0}\f{\overline{A_F(n_1,n_2)}}{n_1n_2}S(am\overline{N}, \pm n_2;cm/n_1 )\,\Omega_\pm\lf( \f{n_2n^2_1X}{c^3Nm};w\ri).\end{split}\ene
\item If $N|c$, we, however, have
\bee\label{34534vof35355}\begin{split}&\sum_{n>0} A_F(m,n)\e{ \f{n\overline{a}}{c}}w\lf(\f{n}{X}\ri) \\
& \hskip 1.2cm =c\sum_{\pm }\sum_{n_1|cm}\sum_{n_2>0}\f{\overline{A_F(n_1,n_2)}}{n_1n_2}S(am, \pm n_2;cm/n_1 )\,\Omega_\pm\lf( \f{n_2n^2_1X}{c^3m};w\ri).\end{split}\ene
\end{itemize}
\end{lemma}

One has a good control of the asymptotic behavior of $\Omega_\pm$, upon expressing it in term of the oscillatory integral; see, for example, \cite[Lemma 6.1]{LX2}.
\begin{lemma}\label{3453464lljyrdfv432} For any fixed integer $\mathcal{K}\ge 1$ and $x\gg 1$, we have
 \[\begin{split}&\Omega_+(x;w)=x\int_{\R^+} w(y) \sum_{j=1}^{\mathcal{K} } \fone{(xy)^{{j}/{3}} }\lf  [ c_{j} \,  e{\lf (3(xy)^{1/{3}} \ri )} \right.\\
&\phantom{=\;\;}\left.   \hskip 5.73cm    
  +d_{j} \,  e{\lf (-3(xy)^{1/{3}} \ri )}  \ri  ] 
  \ud y+O{\lf( x^{({-\mathcal{K}+2})/{3}}\ri )},\end{split}\] where $c_j,d_j$ are suitable constants depending on the four parameters $\alpha_i$, $1\le i\le 3$, above. Furthermore, $\Omega_-(x;\omega)$ has the same expression except values of constants.
\end{lemma}
This, in turn, reduces the right-hand side of the Vorono\u{\i} formula \eqref{34534vof3534} essentially to
\begin{align}\label{20202020030000000000} &\f{X^{{2}/{3}  }}{cN^{1/6}m^{2/3}}\sum_{\pm, n_1|c}\sum_{n_2>0}\f{\overline{A_F(n_1,n_2)}}{ (n_2/n_1)^{1/{3}}}\, S(am\overline{N}, \pm n_2;cm/n_1 ) \notag \\
&\hskip 6.3cm \int_{\R ^+} w(y) \,e\lf (\pm  \f{3\lf(n_2n^2_1Xy/m\ri)^{1/{3}}}{c}\ri) \ud y.\end{align} The main contribution occurs when $n_2\ll c^3m /n^2_1X^{1-\varepsilon}$, by repeated integration by parts, and in which case $\Omega_\pm \ll X^{\varepsilon}$.
\subsection{Non-linear exponential sums}We record the following square-root cancellation estimate for the non-linear exponential sums related to Hecke-Maa\ss\ forms on $\text{GL}(3)$; see \cite[Theorem 3]{RY} and the remarks after \cite[Eqn. (3.18)]{RY}. 
\begin{lemma}\label{0l5y43}
Let $X\ge 2$. Let $W$ be a smooth function supported on $ [1/2, 5/2] $ with partial derivatives satisfying $x^i
W (x/X) \ll_i Z P^i  $ for any $i\ge 0$ and some $Z, P>0$. Then, for any $\alpha \in \R$ with $\alpha^2 \in \Q$, we have 
\[\sum_{m \ge  1} \f{A_F(m,1)}{\sqrt{m}} e(\alpha \sqrt{m})\, W{\lf (\f{m}{X}\ri)}\ll  ZP(1+|\alpha|\log X).\]\end{lemma}
\subsection{Petersson trace formula} For any integers $m,n\ge 1$, we set  
\[\Delta _{k,M}(m,n) = \sum_{f\in \mathcal{B}_k(M)}\omega^{-1}_f \psi_f(m)\psi_f(n)\]
and 
\[\Delta^\ast _{k,M}(m,n) =\sum_{f\in \mathcal{B}^\ast_k(M)}\omega^{-1}_f  \lambda_f(m)\lambda_f(n),\]
where the spectral weights $\omega_f$ are given by
$\omega_f:=\f{(4 \pi)^{k-1}}{\Gamma(k-1)}  \langle   f,f \rangle $. In particular, if $f\in \mathcal{B}^\ast_k(M)$, one has $\omega_f =\f{(k-1)M}{2 \pi^2}L(1,\text{sym}^2 f),$ which satisfies that
$ k^{1-\varepsilon} M\ll \omega_f \ll k^{1+\varepsilon} M $
for any $\varepsilon>0$; see \cite[Lemma 2.5]{ILS}.
 \begin{lemma} \label{20191001} (Petersson trace formula) We have 
 \[ \Delta_{k,M}(m,n)=\delta(m,n)+2 \pi i^{-k} \sum_{c>0} \f{S(m,n;cM)}{cM}J_{k-1} \lf (\f{4 \pi  \sqrt{mn}}{cM}\ri),\] 
where $\delta(m,n)=1$ if $m = n$, and $\delta(m,n)= 0$ otherwise.
\end{lemma}For the purpose of our application, we need an expression for $\Delta^\ast_{k,M}$ instead of $\Delta_{k,M}$.  Suppose that $m,n$ are positive integers satisfying $(m,M)=1$, we have the following identity (see \cite[Eqn. (2.51)]{ILS}):
\begin{lemma} \label{201910000011111}We have 
\[\Delta^\ast _{k,M}(m,n)=\sum_{RS=M}\f{\mu(S)}{L \nu((m,S))} \sum_{l| S^\infty } l^{-1} \sum_{l^2_1|(n, \,Sl_1)} \mu(l_1)l_1  \Delta_{k,R}\lf (m l^2,\f{n}{l^2_1} \ri ) ,\]
where $\nu(M)=[\Gamma(1):\Gamma(M)]=M\prod_{p|M}(1+p^{-1})$.
\end{lemma}
\subsection{Bounds on bilinear forms with Kloosterman sums}The bilinear form with Kloosterman sums always plays a vital r\^{o}le in many topics such as the moments of central values of $L$-functions, subconvexity even the classical bilinear exponential sums in additive number theory; see, for example, \cite{KMS,KSWX,Shp} for descriptions in different regards. In the present paper, we will develop the following relevant results to facilitate our use after a while.
\begin{lemma}\label{c452t45t21} \label{201910000011111}Let $X,Y\ge 1$. Let $W$ be a smooth function supported on $ [1/2, 5/2] \times [1/2, 5/2]$ with partial derivatives satisfying $x^i y ^j
 \f{\partial ^i}{\partial x^i} 
  \f{\partial ^j} {\partial y^j} W(x/X, y/Y) \ll_{i,j} Z Z^i_1 Z^j_2  $ for some $Z, Z_1,Z_2>0$.
Let $q\ge 1$ be an integer. For any $h\in \Z$, we have 
\begin{align} \label{xwxwf34}
\sum_{n\ge 1}\sum_{m\ge 1}\f{A_F(m,1)}{\sqrt{m}}  S(nh,m;q)\,W{\lf (\f{n}{X},  \f{m}{Y},\ri)} \ll Z_1Z\sqrt{XYq}+Z_1ZX q^{3/4}(h,q)^{1/4}+Z_1Z q.
\end{align}
\end{lemma}
\begin{proof}This lemma essentially follows from the work of Kerr, Shparlinski, Wu and Xi \cite{KSWX} which appealed to the approach in 
\cite[Theorem 2.1]{Shp}. By partial summation, one finds that
\[\begin{split}&\sum_{n\ge 1}\sum_{m\ge 1}\f{A_F(m,1)}{\sqrt{m}}  S(n,hm;q)W{\lf (\f{n}{X},  \f{m}{Y}\ri)}\\
&\hskip 3cm \ll \int^{5X/2}_{X/2} \lf |\smu_{x\bmod q}\sum_{m\ge 1}\alpha_{m,t} \, e{\lf (\f{m  \overline{x}}{q}\ri)}
\sum_{n\le t}  e{\lf (\f{nhx}{q}\ri)}
 \ri| \ud t,\end{split}\] with $\alpha_{m,t}=\f{\partial}{\partial t} W(t/X, m/Y)A_F(m,1)/\sqrt{m}$. This is compatible with the setting as presented in the beginning of the \cite[Section 5]{Shp}. We still might bound the inner triple sum by $\sum_{\pm} \sum_{1\le i \le I} |S^\pm_i|$, where, for $1\le i\le I$, $S^\pm_i=\sum_{m\le X} \alpha_{m,t} \gamma_x \, e{ ({\pm m \overline{x}}/{q})}$; see \cite[Section 6]{KSWX} for details.
\end{proof}
Observe that the bilinear forms we are concerned about is a smooth-weighted sum. A direct use of the Possion might show some bonus. Indeed, we alternatively have 
\begin{lemma} \label{201910000011svffs111}Keep the notation as in Lemma \ref{c452t45t21}. Then, for any $(h,q)=1$, we have
\begin{align}\label{xwxwf3433}
\sum_{n\ge 1}\sum_{m\ge 1}\f{A_F(m,1)}{\sqrt{m}}  S(nh,m;q)W{\lf (\f{n}{X},  \f{m}{Y}\ri)} \ll Z\sqrt{XYq}+\sqrt{Z_2}ZX q^{3/4} \mathbf{1}_{ Z_2 q>Y}+ ZX\sqrt{Y}.
\end{align}
\end{lemma}
\begin{proof}By switching the orders of the summations and applying the Cauchy-Schwarz inequality followed by the Possion, one deduces that the double sum is controlled by 
\[\begin{split}&\ll  \lf ( \f{Y}{q}\sum_{n_1,n_2\ge 1}\,\sum_{|\jmath|\ll  Z_2q/Y^{1-\varepsilon}}  \Phi(\jmath;n_1,n_2)\sum_{\alpha \bmod q}S(n_1h, \alpha;q) \overline{S(n_2h,\alpha;q) }  \,e{\lf ( \f{  \jmath \alpha }{q} \ri)} \ri )^{1/2},\end{split}\] with $\Phi(\jmath;v_1,v_2)=\int_{\R^+} W(v_1,t )W(v_2,t )e(-\jmath t/q) \ud t$ for any $v_1,v_2\in \mathbb{R}^+$ and $\jmath\in \mathbb{R}$. It is verifiable that the zero-frequence shows a quantity $Z\sqrt{XYq}+XZ\sqrt{Y}$; while, if $Z_2 q>Y$, the non-zero frequences, however, will contribute the quantity $\sqrt{Z_2}Z q^{3/4}\sqrt{Y}$ by invoking the estimate for triple sums of Kloosterman sums; see, for example, \cite[Lemma 4.1]{KSWX}.
\end{proof}
It is remarkable that, as it can be seen in \eqref{xwxwf3433}, the upper-bound becomes smaller, as $m,n$ vary in the ranges with relatively small heights.
\subsection{Bessel functions} For any $x\ge 1$ and $k\in \mathbb{N}^+$, usually one has the crude bound for the Bessel function $J_{k}(x)$ that
\bee \label{9909909800099}J_{k}(x)\ll {\min {\lf (x^{k},  x^{-1/{2}}\ri)}} ;\ene see, for example, \cite[Proposition 8]{HM}. 
For any $x\in \R$, one might also write
\bee \label{c4c254c} J_{k-1}(x)=e(x)\mathscr{W}^+_{x}(s)+e(-x) \mathscr{W}^-_{k}(x),\ene with\[
\mathscr{W}^\pm_{k}(x)=\f{e^{ \mp \lf (k/2+1/{4}\ri)\pi i}  }{\sqrt{2\pi x}\,\Gamma\lf(k+1/{2}\ri)}\int_{\R^+}e^{-t}\lf ( x\lf ( 1\pm \f{it}{2x}\ri)\ri)^{k-1/{2}}\ud t
,\]
which satisfy that \bee \label{20172002} x^j{\mathscr{W}^\pm_{k}}^{(j)}(x) \ll_{j,k} \min\lf(\fone{\sqrt{x}}, x^{k-1}\ri) \ene for any $j\ge 0$; see, for example, \cite[Sections 16.12, 17.5 \& 16.3]{WW}. Now, for any $h$, a smooth bump function supported on $[1/2, 5/2]$ with bounded derivatives, we define
\bee \label{x24ro226787}\widecheck{h}(t) = \f{1}{\sqrt{2 \pi}} \int_{\R^+} \f{h(\sqrt{u})}{\sqrt{u}}\, e^{itu}\ud u. \ene We thus have (see \cite[Corollary 8.2]{ILS}):
\begin{lemma} \label{201910000t666011svffs111}For any $K>0$ and $x>0$, there holds that
\bee \label{xwx0987wf34331}
\sum_{k \equiv 0 \bmod 2 }i^{-k} h{\lf (\f{k-1}{K}\ri)}J_{k-1}(x)=-\f{K}{2\sqrt{x}}\, \text{Im}{\lf ( e{\lf ( \f{x}{2\pi}-\fone{8}\ri)}\ri)} \widecheck{h} {\lf ( \f{K^2}{2x}\ri)}+O{\lf (\f{x}{K^4}\ri)},
\ene 
where the implied constant depends only on $h$.
\end{lemma} This shows that the oscillatory range for $J_{k-1}(x)$ is roughly $x\gg k^{2+\varepsilon}$, when averaging over the weight $k$; within this range, $J_{k-1}(x)$ obeys the asymptotic behavior as shown in \eqref{xwx0987wf34331}. The function $\widecheck{h}$ behaves like a Schwarz function, and satisfies that ${\widecheck{h}}^{(j)}(x)\ll (1+|x|)^{-A}$ for any large $A$ and $j>0$ by repeated integration by parts sufficiently many times. On the other hand, if there does not exist the a factor $i^k$ in \eqref{xwx0987wf34331}, one, however, has (see \cite[Proposition 8.1]{ILS}):
\begin{lemma} \label{20191000s111}For $\iota =0,2$ and any $K>0$ and $x>0$, there holds that
\begin{align} \label{xwx0987wf3433}\begin{split}
4\sum_{k \equiv \iota  \bmod 4 }h{\lf (\f{k}{K}\ri)}J_{k-1}(x)=&h{\lf (\f{x}{K}\ri)}  -i^{-\iota}\f{K}{\sqrt{x}} \, \text{Im}{\lf ( e{\lf ( \f{x}{2\pi}-\fone{8}\ri)}\ri)} \widecheck{h} {\lf ( \f{K^2}{2x}\ri)}\\
&+O{\lf (\f{x}{K^3}\lf |h^{(3)}{\lf (\f{x}{K}\ri)}\ri |\ri)} +O{\lf (\f{x}{K^4}\ri)}, \end{split}
\end{align} where the implied constants depend only on $h$.
\end{lemma} 
\begin{rem}\label{cf34334t48033}Observe that the denominator $2 \pi $ in both exponential factors of \eqref{xwx0987wf34331} and \eqref{xwx0987wf3433} usually looks inconspicuous in practical use. However, it plays a miraculous r\^{o}le in this paper, and would make it possible for us to apply the square-root estimate in Lemma \ref{0l5y43}. Otherwise, one cannot obtain the Lindelöf average bound in \eqref{009koieu3} whatsoever; see \eqref{x3r3209c323114} and \eqref{kjci34c42} below for details. \end{rem}
\section{Proofs of Theorem \ref{3243251515} and Corollary \ref{corodoe923}}In  this section, we are ready to deduce Theorem \ref{3243251515}, by which one can further show Corollary \ref{corodoe923}. 
\subsection{Proof of Theorem \ref{3243251515}}Let $q$ be a prime and $M$ is also, satisfying that $(q,M)=1$. For any $f\in \mathcal{B}^\ast_k(M)$ and $F$ being a normalized $\text{GL}(3)$ Hecke-Maa\ss\ form of level $q^2$ with trivial nebentypus, we will consider the first moment of the central value $L(1/2,F\otimes f )$ in the levels aspects. Recall that $\mathcal{Q}_{F,f}\asymp M^3q^4 $. The standard functional equation method (see, for example, \cite[Theorem 5.3]{IK}) implies that the central value $L(1/2,F\otimes f)$ hinges on the following two quantities\[ \begin{split}&\sum_{n\ge 1} \f{A_F(n,1)\lambda_f(n)}{\sqrt{n}} U{\lf (\f{n}{q^2M^{3/2}}\ri)} \quad  \text{and}\quad \varepsilon(F\otimes f)\sum_{n\ge 1} \f{A_{\widetilde{F}}(n,1)\lambda_f(n)}{\sqrt{n}} V{\lf (\f{n}{q^2M^{3/2}}\ri)},\end{split} \] up to two smooth weight functions $U, V$ satisfying $U^{{j}}(x), {V}^{{j}}(x)\ll (1+|x|)^{-A}$ for any $j>0$ and sufficiently large $A>0$. Let $L\ge 1$ be a parameter to be optimized later, satisfying that \bee \label{3x22r43} L< M^{1/8} \ene  and 
 \bee \label{3x22r431}   L\ge  \f{q^{1+\varepsilon}} {{M^{1/4}} } , \quad L\ge  \f{{M} ^{1/4}}{{q^{1-\varepsilon}} } \ene for any $\varepsilon>0$. We now define a set of modulo 
\bee \label{nxr4300509454}\mathscr{P}_L=\{p : p \text{ is prime, }L\le p\le 2L, (p,qM)=1\} .\ene Let  $f_0\in \mathcal{B}^\ast_k(M)$. By equipping with the amplifier (as in \cite[Section 12.2]{BK19}) 
\bee \label{3c5vt24yv}\mathfrak{A}_f=\sum_{j=1}^2 \lf| \sum_{\ell \in \mathscr{P} _L}  \lambda_f(\ell^j) x(\ell^j) \ri|^2, \ene with $x(n)={\text{sgn}}(\lambda_{f_0}(n))$, one attempts to estimate an amplified first moment 
\[\sum_{f\in \mathcal{B}^\ast_k(M)}\omega^{-1}_f  \mathfrak{A}_f L(1/2, F\otimes f).\]
It is clear that, by the relation $\lambda^2_{f_0}(p)-\lambda_{f_0}(p^2)=1$, one sees that \bee\label{x414rx310993}
\mathfrak{A}_{f_0} \gg\fone{2} \lf (\sum_{\ell \in \mathscr{P}_L} |\lambda_{f_0}(\ell)|+|\lambda_{f_0}(\ell^2)|^2\ri)^2\gg \f{ L^2}{\log L}.\ene If one opens the square and applies
multiplicativity of Hecke eigenvalues, finally our main focus is on
\bee \label{3cr24t2t42vc}\sum_{f\in \mathcal{B}^\ast_k(M)}\omega^{-1}_f  \lambda_f(\ell) L(1/2, F\otimes f),\ene in view of the fact that 
\begin{align}\label{cq4rr3r3}\mathfrak{A}_{f}=\sum_{\ell \in \mathscr{P}_L} \lf (x^2(\ell) +x^2(\ell^2) \ri) &+ \sum_{\ell_1,\ell_2 \in \mathscr{P}_L}  x(\ell^2_1) x(\ell^2_2)\lambda_f(\ell^2_1\ell^2_2)\notag \\
&+ \sum_{\ell_1,\ell_2 \in \mathscr{P}_L}  \lf ( x(\ell_1) x(\ell_2)+x(\ell^2_1) x(\ell^2_2)\mathbf{1}_{\ell_1=\ell_2} \ri) \lambda_f(\ell_1\ell_2).\end{align}Notice that, here and thereafter, \bee \label{x4cr31466}\ell =\{1,\ell_1\ell_2, \ell^2_1\ell^2_2\},\ene with $\ell_1,\ell_2\in \mathscr{P}_L$, is either 1 or a square-free integer or a square satisfying the coprimalities with the parameters $q$ and $M$, upon recalling \eqref{cq4rr3r3}.

 To proceed, let us recall that the root number $\varepsilon(F\otimes f)=\pm  \prod_{p|q}\varepsilon(\Pi_{F,p})^2$ according to the foregoing descriptions in \S 2.3. In the $`+'$ case, one finds that, the quantity in \eqref{3cr24t2t42vc} is boiled down to the analysis of the following sum
\bee \label{x33} 
\sum_{f\in \mathcal{B}^\ast_k(M)}\omega^{-1}_f \lambda_f(\ell) \sum_{n\ge 1} \f{A_F(n,1)\lambda_f(n)}{\sqrt{n}} U{\lf (\f{n}{\mathcal{X}}\ri)},
\ene
with $ \mathcal{Q}^{(1-\delta)/2}_{F,f} \le  \X\le \mathcal{Q}^{(1+\varepsilon)/2}_{F,f}$ for any small constants $\delta,\varepsilon>0$. Here, $U$ is a smooth bump function supported $[1/2, 5/2]$, identical to one for $x\in [1,2]$, with derivatives satisfying that $x^j U^{(j)}\lf ({x}/{\X} \ri) \ll_j 1$ for any $j\ge 0$.

 By the Cauchy-Schwarz inequality, together with the Rankin-Selberg bounds for the second moments of Fourier coefficients, the trivial estimate for the double sum in \eqref{x33} is $\ll_\varepsilon \X^{1/2+\varepsilon}$. This is far away from the bound $ \ll \X^{1/2+\varepsilon}/M$ leading to the convexity. The task is thus reduced to estimating \eqref{x33} further. Our strategy next is the Petersson trace formula. To set up for our application of this, we trivially write the expression \eqref{x33} above as  \bee \label{209093094000994}\sum_{\substack{n\ge 1\\(n,M)=1}}   \f{A_F(n,1)  }{\sqrt{n}}\Delta^\ast_{k,M}(n,\ell) U {\lf (\f{n}{\mathcal{X}} \ri)}+\sum_{\substack{n\ge 1\\M|n}}   \f{A_F(n,1) }{\sqrt{n}}\Delta^\ast_{k,M}(n,\ell) U {\lf (\f{n}{\mathcal{X}} \ri)}.\ene It is straightforward to verify that, by the fact $|\lambda_f(M)|= M^{-1/2}$, the second term is bounded by $\ll_\varepsilon \sqrt{\mathcal{X}}/M^{3/2}$. For $(n\ell,M)=1$, notice that, by Lemma \ref{201910000011111} one has the following relation 
\[ \Delta^\ast_{k,M}(n,\ell)=\Delta_{k,M}(n,\ell)-\fone{M}\sum_{\rho |M^\infty} \fone{\rho }\Delta_{k,1}{\lf(n, \ell \rho^2\ri)},\]
 so that one may plug the second term on the right-hand side above into \eqref{209093094000994} showing that  \[\fone{M}\sum_{\substack{\rho|M^\infty \\ \rho >1}} \fone{\rho}\sum_{\substack{n\ge 1\\(n,M)=1}}\f{A_F(n,1)\lambda_f(n)}{\sqrt{n}}\Delta_{k,1}{\lf(n,\ell \rho ^2 \ri)}  U {\lf (\f{n}{\X} \ri)}\ll_\varepsilon \f{\sqrt{\X}}{M^{2-\varepsilon}}.\]
This reduces \eqref{x33} to
\bee \label{99949989099949043909}  \sum_{n\ge 1} \f{A_F(n,1) }{\sqrt{n}}\lf \{ \Delta_{k,M}(n,\ell) -\fone{M}\Delta_{k,1}(n,\ell) \ri \} U {\lf (\f{n}{X} \ri)} \ene
plus an acceptable error $O_\varepsilon {(\X^{1/2+\varepsilon}/M^{3/2})}$. Here, the first term in the braces should be viewed as an average over all forms of level $M$. The second, as an average over the old forms, however, will not contribute significantly compared with the newforms. Indeed, in this case one finds that \[\begin{split}&\fone{M} \sum_{n\ge1} \f{A_F(n,1)}{\sqrt{n} }  \Delta_{k,q}(n,1)  U {\lf (\f{n}{X}\ri)}
=\fone{M}  \sum_{f^\star \in \mathcal{B}^\ast_{k}(1)} \omega^{-1}_{f^\star}\sum_{n\ge1} \f{A_F(n,1)\lambda_{f^\star }(n)}{\sqrt{n} }   U {\lf (\f{n}{X}\ri)},\end{split}\]
which is $\ll q^{1+\varepsilon}/M$ by the convexity
bound $O(q^{1+\varepsilon})$ for the Rankin-Selberg $L$-function $ L(s, F\otimes f^\star)$. Our job thus reduces to
 \bee \label{3c23t252}
\sum_{n\ge 1} \f{A_F(n,1) }{\sqrt{n}}\Delta_{k,M}(n,1) +O{\lf(\f{\sqrt{\X}}{M^{3/2-\varepsilon}}  +\f{q^{1+\varepsilon}}{M}\ri)}.\ene

Now, we are ready to apply the Petersson trace formula 
  (Lemma \ref{20191001}) to $\Delta_{k,M}$; this produces the diagonal and off-diagonal terms of the form $\mathcal{O}(M,q)+2\pi i^{k} \mathcal{S}(M,q)$, where  \bee \label{49999499999909149}\mathcal{O}(M,q)= \sum_{n=1}\f{A_F(n,1)\lambda_f(n)}{\sqrt{n}}\delta(n,\ell)U {\lf (\f{n}{\mathcal{X}} \ri)}\ene which is $O(\ell^{-1/2+\varepsilon})$, and \bee \label{900999990004949991}\mathcal{S}(M,q)=\sum_{n\ge 1}\f{A_F(n,1 )}{\sqrt{n}} U {  \lf (\f{n}{\mathcal{X}} \ri)   }\sum_{c\ge 1}\frac{ S(n,\ell;cM)}{cM} J_{k-1}\lf (\f{4\pi \sqrt{n\ell}}{cM}\ri). \ene
 Noticing that we are taking $k$ sufficiently large throughout the paper, one finds that the sum over $c$ can be truncated at $c\ll \sqrt{\mathcal{X}\ell }/M^{1-\varepsilon}$; the contribution in the complementary range is negligibly small (i.e., $O_A(\X^{-A})$ for any sufficiently large $A$), with the help of the bound \eqref{9909909800099}. It is instructive to see that this sum of $c$ is not vacuous by the second pre-condition in \eqref{3x22r431}; otherwise, one cannot succeed in obtaining hybrid subconvexity). In particular, whenever $\ell=1$, it requires that $q\ge M^{1/4+\varepsilon}$ (so that $\sqrt{\X}\gg M^{1+\varepsilon}$ for any $\varepsilon>0$), otherwise the display \eqref{900999990004949991} is negligibly small; we will henceforth assume this pre-condition once we are considering the case of $\ell=1$ in the following paragraphs.

 By using the dyadic subdivision, we further attempt to estimate the multiple sum for \eqref{900999990004949991} as follows
\bee \label{3r34j66633s}\sup_{C\ll \sqrt{\mathcal{X}\ell }/M^{1-\varepsilon}} \sum_{n\ge 1}\f{A_F(n,1 )}{\sqrt{n}} \sum_{c\sim C}\frac{ S(n,\ell;cM)}{cM} \widetilde{U}_\ell\lf (\f{n}{\mathcal{X}}\ri) ,\ene  where ${U}_\ell\lf (y\ri)=U(y)J_{k-1}\lf ({4\pi \sqrt{y \mathcal{X}\ell}}/{CM}\ri)$ for any $y\in \R^+$. According to the properties of the $J$-Bessel functions in \S2.1, it follows that \bee \label{xv565656} y^i \ell ^j
 \f{\partial ^i}{\partial y^i} 
  \f{\partial ^j} {\partial \ell^j} U_{\ell}(y)\ll \lf (  \f{\sqrt{\X L}}{CM}  \ri)^{i-1/2}  \ene for any $i,j\ge 0$.

  Next, we prepare to apply the Vorono\u{\i} summation formula, Lemma \ref{2017303534623421}. Before that, one needs to detect the coprimality between $\ell$ and the modulus of the twist $cM$. Recall that $\ell =\{1,\ell_1\ell_2, \ell^2_1\ell^2_2\}$ given as in \eqref{x4cr31466}. 
  
  \emph{Case I: $\ell=\ell_1\ell_2$.} To begin with, let us consider what happens whenever $\ell=\ell_1\ell_2$. It is verifiable that the only counterexample with $( c,\ell)\neq 1$ occurs, provided that $q>M^{1/4+\varepsilon}$, and in which case, $c=\ell_1c^\prime$ or $\ell_2c^\prime$ for some $c^\prime \ll \sqrt{\X}/M^{1-\varepsilon}$, with $(c^\prime, \ell_1\ell_2M)=1$ by first one of the restrictions in \eqref{3x22r431}. If so, one finds that $S(n, \ell_1\ell_2; \ell_1 c^{\prime} M)=-S(n,\ell_2 \overline{\ell_1};c^{\prime} M)$, provided that $(n, \ell_1)=1$, and $S(n, \ell_1\ell_2; \ell_1 c^{\prime} M)=(\ell_1-1)S(n/\ell_1,\ell_2;c^{\prime} M)$ otherwise, where $\ell_1|n$. Our object \eqref{3r34j66633s} is essentially degenerated into two terms 
  \begin{align} 
  \fone{L}&\sup_{C\ll \sqrt{\mathcal{X} }/M^{1-\varepsilon}} \sum_{n\ge 1}\f{A_F(n,1 )}{\sqrt{n}} \sum_{c^\prime\sim C}\frac{ S(n,\ell_2 \overline{\ell_1};c^\prime M)}{c^\prime M} \widetilde{U}_\ell\lf (\f{n}{\mathcal{X}}\ri)  \label{x4c42bjjde3j},\\
\f{A_F(\ell_1,1)}{\sqrt{L}}&\sup_{C\ll \sqrt{\mathcal{X} }/M^{1-\varepsilon}} \sum_{n\ge 1}\f{A_F(n,1 )}{\sqrt{n}} \sum_{c^\prime\sim C}\frac{ S(n,\ell_2 ; c^\prime M)}{c^\prime M} \widetilde{U}_\ell\lf (\f{n}{\mathcal{X}/L}\ri)    \label{x4c42b62jjj};
  \end{align} whilst, the expression of \eqref{3r34j66633s} we are concerned about is exactly 
\bee\label{00dp30r4t}\sup_{C\ll \sqrt{\mathcal{X}}L/M^{1-\varepsilon}} \sum_{n\ge 1}\f{A_F(n,1 )}{\sqrt{n}} \sum_{c\sim C}\frac{ S(n,\ell_1\ell_2;cM)}{cM} \widetilde{U}_\ell\lf (\f{n}{\mathcal{X}}\ri) .\ene It suffices to evaluate \eqref{00dp30r4t} with $(c,\ell_1\ell_2)=1$, observing that the expression \eqref{00dp30r4t} bears some resemblances to that in \eqref{x4c42bjjde3j} and \eqref{x4c42b62jjj}. By a similar argument, the two possibly-existing quantities \eqref{x4c42bjjde3j} and \eqref{x4c42b62jjj} will provide a relative less significant contribution. For the sake of having a more clear understanding of the two quantities above for readers, we re-affirm that these two terms will not come into being, whenever $q<M^{1/4}$. 

 \emph{Case II: $\ell=\ell^2_1\ell^2_2$.} Now, we turn to examining that the case of $\ell$ being a square in \eqref{3r34j66633s}. Observe that, the (possibly occurred) scenarios where $( c,\ell)\neq 1$ are as follows:
\begin{itemize}
\item [(a)]  $\ell=\ell^2_1\ell^2_2, c =c^\ast \ell_1 $ or $c^\ast \ell_2 $ for some $c^\ast\ge 1$, such that $ (c^\ast,\ell_1\ell_2M)=1$ and $c^\ast\ll \sqrt{\X}L/M^{1-\varepsilon}$;
\item [(b)] $\ell=\ell^2_1\ell^2_2, q>M^{1/4+\varepsilon}$, and $c =\hat{c} \ell_1\ell_2 $  for some $\hat{c}  \ge 1$, satisfying that $ (\hat{c} ,\ell_1\ell_2M)=1$ and $\hat{c}\ll \sqrt{\X}/M^{1-\varepsilon}$.
\end{itemize}
Notice that, here, $\ell^2_1\ell_1^2|c $ and $\ell^2\ell^2_2|c $ can not occur, in view of the first constraint in \eqref{3x22r431}. By a similar argument as before, one finds that, in those two cases, (a) and (b), the sum in \eqref{3r34j66633s} is thus degenerated into 
\begin{align} \f{1}{L} &\sup_{C\ll \sqrt{\mathcal{X} }L/M^{1-\varepsilon}} \sum_{n\ge 1}\f{A_F(n,1 )}{\sqrt{n}} \sum_{c^\ast\sim C}\frac{ S(n,\ell^2_2 ;c^\ast  M)}{
c^\ast M} \widetilde{U}_\ell\lf (\f{n}{\mathcal{X}}\ri)\label{xr2fkkd6},\\ \f{A_F(\ell_1,1)}{\sqrt{L}} &\sup_{C\ll \sqrt{\mathcal{X} }L/M^{1-\varepsilon}} \sum_{n\ge 1}\f{A_F(n,1 )}{\sqrt{n}} \sum_{c^\ast\sim C}\frac{ S(n,\ell_1\ell^2_2 ;c^\ast M)}{
\hat{c}M} \widetilde{U}_\ell\lf (\f{n}{\mathcal{X}/L}\ri),\label{xr2fkkd8}\end{align}
and 
\begin{align} 
 \f{1}{L^{2}} &\sup_{C\ll \sqrt{\mathcal{X} }/M^{1-\varepsilon}} \sum_{n\ge 1}\f{A_F(n,1 )}{\sqrt{n}} \sum_{\hat{c}\sim C}\frac{ S(n,1 ;\hat{c} M)}{
\hat{c}M} \widetilde{U}_\ell\lf (\f{n}{\mathcal{X}}\ri)\label{xr2fkkd},
\\
\f{A_F(\ell_1\ell_2,1)}{L}&\sup_{C\ll \sqrt{\mathcal{X} }/M^{1-\varepsilon}} \sum_{n\ge 1}\f{A_F(n,1 )}{\sqrt{n}} \sum_{\hat{c} \sim C}\frac{ S(n,\ell_1\ell_2;\hat{c} M)}{\hat{c} M} \widetilde{U}_\ell\lf (\f{n}{\mathcal{X}/L^2}\ri) , \label{xr2fkkd2}
\end{align} respectively. Recall that, in {\emph{Case II}}, the sum \eqref{3r34j66633s} is explicitly of the following form
\bee \label{vo53653}\sup_{C\ll \sqrt{\mathcal{X}}L^2/M^{1-\varepsilon}} \sum_{n\ge 1}\f{A_F(n,1 )}{\sqrt{n}} \sum_{c\sim C}\frac{ S(n,\ell^2_1\ell^2_2;cM)}{cM} \widetilde{U}_\ell\lf (\f{n}{\mathcal{X}}\ri) .\ene
On account of the similarities between the expression in \eqref{vo53653} and that in \eqref{xr2fkkd6}--\eqref{xr2fkkd2}, it suffices for us to handle the sum \eqref{vo53653} with $(c,\ell_1\ell_2)=1$, akin to {\emph{Case I}}. In a similar vein, one can deal with the four sub-sums \eqref{xr2fkkd6}--\eqref{xr2fkkd2} in the degenerate case, and the contributions from them are dominated by that of \eqref{vo53653}. Upon combining with this and the argument in {\emph{Case I}} as above, our goal eventually is boiled down to estimating \eqref{3r34j66633s} with $(c, \ell)=1$ in the following stage.

  Having seen that the $n$-sum of \eqref{3r34j66633s} is already in a shape for an application of the the Vorono\u{\i}, upon using \eqref{34534vof3534}, one might reduce \eqref{3r34j66633s} essentially to
\bee \label{newvoi1}\begin{split} \sup_{C\ll \sqrt{\mathcal{X}\ell }/M^{1-\varepsilon}} \f{CMq}{\sqrt{\mathcal{X}}}& \sum_{c\sim C}\sum_{\substack{\pm , m,r\\m \ge 1\\
r|cM}}  \f{\overline{A_{F}(m,r)}}{mr^2 } \, e\lf (\mp \f{mr^2 \overline{\ell q^2}}{cM}\ri)\, \Omega_{\pm}\lf ({\f{mr^2 \mathcal{X}}{(cM)^3q^2};{U}_\ell}\ri)
  .\end{split}\ene Here, we have applied the estimate involving Ramanujan sum for each modulus $c$ that $
S(n,0;c)=\sum_{ab=c}\mu(a)\sum_{\beta \bmod c}e{\lf ({\beta n}/{c} \ri) }$. One should notice that the additive twist does not {\emph{collude}} with the level anymore, namely, there holds that $(cM,q)=1$; this, however, is guaranteed by the pre-condition in \eqref{3x22r43}, which confines that $c<q$\footnote[3]{Indeed, on account of the final restrictions for the parameter $L$ (see \eqref{c4r1t55} below), the pre-condition, here, seems dispensable (as far as the establishment of effective ranges for subconvexity is concerned).}.  By the asymptotic formula, Lemma \ref{3453464lljyrdfv432}, for the integral transform $\Omega_{\pm}$, we are further reduced to the following superiorem problem that
\bee \label{ooo9944}\begin{split} \mathcal{S}(M,q) \ll \sup_{1\le C\ll \sqrt{\mathcal{X}\ell }/M^{1-\varepsilon}} \Psi (C; \ell, M, q) ,\end{split}\ene  with 
  \[\Psi (C; \ell, M, q)= \f{\X^{1/6}}{C M q^{1/3}}\,
\sum_{c\sim C}\sum_{\substack{\pm , m,r\\m \ge 1\\
r|c M}}  \f{\overline{A_{F}(m,r)}}{m^{1/3}r^{2/3} } \, e\lf (\mp \f{mr^2 \overline{\ell q^2}}{c M}\ri)\, \widetilde{{U}_\ell}\lf ({\f{mr^2 \mathcal{X}}{(c M)^3q^2}}\ri).
  \]
Here, $\widetilde{{U}_\ell}(x)=\int_{\R^+} U_{\ell }(y)e(\pm 3(xy)^{1/3})\ud y$ for any $x\in \R^+$. By the feature \eqref{c4c254c}, it can be expressed roughly as
\bee \label{c35yub4u444}\lf ( \f{\sqrt{\X \ell}}{CM}\ri)^{-1/2}\int_{\R^+} U(y)F^+_k{\lf ( \f{4\pi \sqrt{y \mathcal{X}\ell}}{CM} \ri)} \,e{\lf( \f{4\pi \sqrt{y \mathcal{X}\ell}}{CM} \pm 3(xy)^{1/3} \ri)}\ud y\ene for any $x\in \R^+$, where $F^
+_k(t)=\sqrt{2\pi t}\, \mathscr{W}^
+_k(t)$ for any $t\in \R$, satisfying that $t^j {F^
+}^{(j)}_k(t) \ll _{j,k} 1$ for any $t\gg 1$ by \eqref{20172002}.
\subsubsection{Non-oscillatory case}
In this subsection, we would firstly proceed our argument in the situation where the weight $U_\ell(y)$ is of the phase roughly ``1''. In other words, we start by working with $C$ localizing at $C\asymp \sqrt{\X \ell}/M^{1-\varepsilon} $ in \eqref{ooo9944}. Our target thus simplifies into
\[\begin{split} \sup_{C\asymp \sqrt{\mathcal{X}\ell }/M^{1-\varepsilon}}  \f{\X^{1/6}}{C M q^{1/3}}
& \sum_{c\sim C}\sum_{\substack{\pm , m,r\\m \ge 1\\
r|c M}}  \f{\overline{A_{F}(m,r)}}{m^{1/3}r^{2/3} }  \, e\lf (\mp \f{mr^2 \overline{\ell q^2}}{c M}\ri)\, \widetilde{{U}_\ell}\lf ({\f{mr^2 \mathcal{X}}{(c M)^3q^2}}\ri)
  .\end{split}\]
One finds that $\widetilde{U}^{(j)}_\ell\ll X^\varepsilon $ for any $j>0$, by repeated integration by parts many times. It thus follows that one may 
truncate the variable $m$ at \[m\ll \X^\varepsilon+q^2(CM) ^3/\mathcal{X}^{1-\varepsilon} \ll q^2 \sqrt{\mathcal{X}} \ell^{3/2+\varepsilon}, \] at the cost of a negligible error. Our next strategy is to make use of the reciprocity law, followed by employing the \text{GL}(3)-Vorono\u{\i}; this would lead to a dual form of \eqref{3r34j66633s}. To be precise, one writes 
\[e\lf ( \f{mr^2 \overline{\ell q^2}}{c M}\ri)=e\lf ( -\f{mr^2 \overline{c M}}{\ell q^2 }\ri)e\lf ( \f{mr^2}{c M\ell q^2}\ri).\] Observe that the harmonic $e\lf ( {mr^2}/{c M\ell q^2}\ri)$ is automatically {\emph{flat}}. Indeed, this is not surprising, since that all of the variables in the Kloosterman sums always lie in the ``Linnik range''. It thus suffices to evaluate 
\bee \label{newvoi2}\begin{split}\sup_{C\asymp\sqrt{\mathcal{X}\ell }/M ^{1-\varepsilon}}  \f{\X^{1/6}}{C M q^{1/3}}
& \sum_{c\sim C}\sum_{\substack{\pm , m,r\\m \ge 1\\
r|c M}}  \f{\overline{A_{F}(m,r)}}{m^{1/3}r^{2/3} } \,  e\lf (\pm \f{mr^2 \overline{c M}}{\ell q^2}\ri)\, \widetilde{{U}_\ell}\lf ({\f{mr^2 \mathcal{X}}{(c M)^3q^2}}\ri)
  .\end{split}\ene
For any given $\mathscr{M}\ge 1$, we introduce a smooth function $\eta_{\mathscr{M}}(m) $ such that $\sum_{\mathscr{M}\ge 1} \eta_{\mathscr{M}}(m)= 1$, where $\eta_{\mathscr{M}}(m) $
is supported on $m\sim \mathscr{M}$, with $\mathscr{M}$ running through the powers of 2 independently, and
satisfies that $\mathscr{M}^j  \eta_{\mathscr{M}}^{(j)} (n)\ll \X^\varepsilon$ for any $j\ge 0$. Thus, if now one defines \[ {\mathcal{W}_{\pm, r,\ell}}{\lf ({\rho}, \nu\ri)}=\eta_{\mathscr{M}}(\rho \mathscr{M})\, \widetilde{U}_\ell\lf ({\f{\rho \mathscr{M} r^2 \mathcal{X}}{(\nu C  M)^3q^2}}\ri)  \]for any ${\rho}, \nu\in \R^+$, then we can recast \eqref{newvoi2} as 
\begin{align}\label{3c42r653} &\sup_{C\asymp\sqrt{\mathcal{X}\ell }/ M^{1-\varepsilon}} \,  \f{\X^{1/6+\varepsilon}}{C M q^{1/3}} 
\,\sup_{\substack{r\le CM\\(r, \ell q)=1}}  \fone{r^{2/3}  }\, \sup_{\mathscr{M}\ll q^2(C M)^3/r\X^{1-\varepsilon}}  \lf |\mathcal{K}_{\pm, \ell}( \mathscr{M} , C,r)\ri |
,\end{align} where
\bee \label{d3d2109dmd}\mathcal{K}_{\pm, \ell} ( \mathscr{M} , C,r)=\sum_{c\sim C}\sum_{\substack{\pm , m \ge 1 }}  \f{\overline{A_{F}(m,r)}}{m^{1/3} }  \, e\lf (\pm \f{mr^2 \overline{c  M}}{\ell q^2}\ri)\, {\mathcal{W}_{\pm, r,\ell}} {\lf (\f{m}{\mathscr{M}};\f{c}{C}\ri)}\ene

We will merely consider the $`+'$ case in the following; the case $`-'$ can be handled in an entirely analogously manner. Observe that the coprimality relation $(r^2\overline{c M},\ell q^2)=1$ has already held by the prior analysis. We now resort to the Vorono\u{\i} formula in the ramified case, \eqref{34534vof35355}, which further converts the double sum $\mathcal{K}_{\pm, \ell} ( \mathscr{M} , C,r)$ essentially into
\[\begin{split}&\f{\ell q^2 }{\mathscr{M}^{1/3}}\sum_{n^\prime |\ell q^2r}\, \sum_{c\sim C} \sum_{\pm, n\ge 1} \f{A_F(n^\prime, n)}{n^\prime n} \,S(r \overline{r^2} cM, \pm n; \ell q^2 r/n^\prime) \,\Omega_\pm {\lf (\f{n{n^\prime }^2 \mathscr{M}}{ \ell ^3 q^6 r}; {\mathcal{W}_{\pm, r,\ell}}
\ri) } ,\end{split}\] where $\overline{r}$ is such that $r\overline{r} \equiv 1 \bmod {\ell q^2}$. Via Lemma \ref{3453464lljyrdfv432}, this, however, is controlled by the upper-bound of
\bee \label{xe1433}\begin{split}& \f{\mathscr{M}^{1/6}}{q\sqrt{\ell r} } \sum_{n^\prime |\ell q^2r}\, \sum_{c\sim C} \sum_{\pm, n\ge 1} \f{A_F(n^\prime, n)}{\sqrt{n}} \, S(r \overline{r^2}  cM, \pm n; \ell q^2 r/n^\prime) \, \widetilde{\mathcal{W}_{\pm, r,\ell}} {\lf (\f{n{n^\prime }^2 \mathscr{M}}{ \ell^3 q^6 r};\f{c}{C}\ri)} ,\end{split}\ene 
 where the integral $\widetilde{\mathcal{W}_{\pm, r,\ell}} $ is given by
\[ \widetilde{\mathcal{W}_{\pm, r,\ell}}(v;\iota ) =\int_{\R^+} {\mathcal{W}_{\pm, r,\ell}} {\lf (y;\iota \ri)}\, e {\lf (   \pm 3(vy)^{1/3}\ri)}\ud y\] for any $v,\iota \in \R^+$. Appealing to \eqref{xv565656}, one temporarily
sees that $ \f{{\partial ^i}  {\partial ^j}}{{\partial v^i}  {\partial \iota^j} } \, \widetilde{\mathcal{W}_{\pm, r,\ell}}(v;\iota) \ll_{i,j} \X^\varepsilon$ for any $i,j,
\varepsilon>0$.

Observe that, unlike other scenarios where one only cares about what happens while the parameter $\mathscr{M}  $ varying around its height by directly adding a smooth weight after the Cauchy-Schwarz (see, for example, \cite{Mu1,HL} for relevant details), we need to optimize $\mathcal{K}_{\pm, \ell} ( \mathscr{M} , C,r)$ for any $1\le \mathscr{M}\ll q^2(CM)^3/r\X^{1-\varepsilon}$ in \eqref{3c42r653}. Next, we will proceed the argument by distinguishing two different situations.

---\emph{When $\mathscr{M}$ is small.} We first consider the case where $\mathscr{M}$ is small such that $1 \le \mathscr{M}\le \mathscr{M}^
\ast$, say, with $\mathscr{M}^
\ast\ll q^2(C M)^3/r\X^{1-\varepsilon}$. Recall \eqref{d3d2109dmd}. We then have 
\bee \label{23er3xr31t31}
\mathcal{K}_{\pm, \ell} ( \mathscr{M} , C,r) \ll  (\mathscr{M}^
\ast)^{2/3} r^{7/32} C^{1+\varepsilon}.
\ene Here, we have used the estimate $|A_F(m,n)|\ll (mn)^{7/32+\varepsilon}$, which follows from the best record towards to the generalized Ramanujan conjecture for ${
\rm{GL}(2)}$ due to Kim and Sarnak; see \cite[Proposition 2]{KS}.

---\emph{When $\mathscr{M}$ is large.} Next, let us turn to the case where $\mathscr{M}$ is suitably large, with \bee \label{kx233r4} \mathscr{M}^
\ast< \mathscr{M}\ll \f{q^2(C M)^3}{r\X^{1-\varepsilon}} .\ene We are in a position to rely on the estimates for bilinear Kloosterman sums securing further cancellations in \eqref{xe1433}. It turns out that the quantity we are really concerned about is the following 
\[\begin{split}& \sup_{\mathscr{M}^\ast \le \mathscr{M}\ll \f{q^2(CM)^3}{r\X^{1-\varepsilon}}} \f{\mathscr{M}^{1/6}}{q\sqrt{\ell r} } \sum_{n^\prime |\ell q^2r}\, \sum_{c\sim C} \sum_{\pm, n\ge 1} \f{A_F(n^\prime, n)}{\sqrt{n}} \, S(r \overline{r^2}  cM, \pm n; \ell q^2 r/n^\prime) \, \widetilde{\mathcal{W}_{\pm, r,\ell}} {\lf (\f{n{n^\prime }^2 \mathscr{M}}{ \ell^3 q^6 r};\f{c}{C}\ri)} .\end{split}\]
It is necessary that $
{ \ell^3 q^6 } \gg {(C M)^3 q^2} /\X^{1-10\varepsilon} $, with $ C\asymp \sqrt{\mathcal{X}\ell }/ M^{1-\varepsilon}$, so that the sum over $n$ is not vacuous. Now, an application of Lemma
\ref {c452t45t21}, with $Z=\X^\varepsilon$ and $Z_1=Z_2=1$, reveals that the display above is dominated in the magnitude by
 \begin{align*}&\ll  \sup_{\mathscr{M}^\ast \le \mathscr{M}\ll q^2(CM)^3/r\X^{1-\varepsilon}}  \sum_{n^\prime |\ell q^2r}\,  \f{\mathscr{M}^{1/6+\varepsilon} (n^\prime)^{7/32}}{q\sqrt{\ell r } }  \lf \{r(\ell q^2)^2 \sqrt{\f{C}{ (n^\prime)^3 \mathscr{M}}} 
  \right.\\
&\phantom{=\;\;}\left.   \hskip5.3cm+\lf(\f{\ell q^2 r}{n^\prime} \ri)^{3/4} (r \overline{r^2}  M, \ell q^2r/n^\prime)^{1/4}C+\f{\ell q^2 r}{n^\prime}\ri \}\\
 &\ll \f{ \sqrt{Cr}(\ell q^2)^{3/2}}{{\mathscr{M}^\ast }^{1/3}}+  \f{\sqrt{Cp} (qr)^{1/3+\varepsilon}}{\X^{1/6}}  
 \lf [q\sqrt{\ell }+ (\ell q^2)^{1/4}C \ri].\end{align*}
 Combining with \eqref{23er3xr31t31}, it follows that 
 \[
\mathcal{K}_{\pm, \ell} ( \mathscr{M} , C,r) \ll  (\mathscr{M}^
\ast)^{2/3} r^{7/32} C^{1+\varepsilon}+\f{ \sqrt{Cr}(\ell q^2)^{3/2+\varepsilon}}{{\mathscr{M}^\ast }^{1/3}}+  \f{\sqrt{Cp} (qr)^{1/3+\varepsilon}}{\X^{1/6}}  
 \lf [q\sqrt{\ell }+ (\ell q^2)^{1/4}C \ri].
\]Upon choosing 
\bee \label{x3r109x}\mathscr{M}^\ast = \f{(\ell q^2)^{3/2}}{\sqrt{C}r^{9/32}},\ene it turns out that
\[
\sup_{\mathscr{M}\ll q^2(C M)^3/r\X^{1-\varepsilon}}  \mathcal{K}_{\pm, \ell}( \mathscr{M} , C,r) \ll  \ell q ^2 C^{2/3+\varepsilon} r^{1/32}+ \f{\sqrt{CM} (qr)^{1/3+\varepsilon}}{\X^{1/6}}  
 \lf [q\sqrt{\ell }+ (\ell q^2)^{1/4}C \ri].\]Inserting this expression into \eqref{3c42r653}, we obtain 
\[\sup_{C\asymp \sqrt{\mathcal{X}\ell }/ M^{1-\varepsilon}} \,   \f{\X^{1/6}}{C M q^{1/3}} \,  \,\sup_{\substack{r\le CM\\(r, \ell q)=1}}  \fone{r^{2/3}  }\,\lf ( \ell q^2  C^{2/3} r^{1/32}+ \f{\sqrt{CM} (qr)^{1/3+\varepsilon}}{\X^{1/6}}  
 \lf [q\sqrt{\ell }+ (\ell q^2)^{1/4}C \ri] \ri ).
\]
Observes that the contribution from the portion for $r$ being large (such that $r\gg \X^\delta,\delta>0 $) is much lesser than that from the dominated case of $r=1$ in the expression above. It suffices to concentrate on this case, and we thus are able to deduce
\begin{lemma} For any $\ell $ given as in \eqref{x4cr31466}, which satisfies that\bee \label{werwgwre}
\ell \ge \f{M^{1/4}}{q^{1-\varepsilon}}, \quad { \ell^3 q^6}\ge \f{(C M)^3 q^2} {\X^{1-\varepsilon }}, \quad 1\ll \f{ (\ell q^2)^{3/2}}{\sqrt{C} }\le  \f{q^2(C M)^3}{\X^{1+\varepsilon}},\ene with $ C\asymp \sqrt{\mathcal{X}\ell }/ M^{1-\varepsilon}$, we have
\bee\label{xw25y} \Psi (C; \ell, M, q) \ll \f{ \ell q^{5/3} \X^{1/6+\varepsilon}}{M}+q^{1+\varepsilon}\sqrt{\f{\ell }{ M}}+
 \f{\ell^2\X^{1/4+\varepsilon}\sqrt{q}}{M}.
\ene
\end{lemma}
Here, the third condition of \eqref{werwgwre} stems from the restriction \eqref{kx233r4} in the dominated case of $r=1$, with the choice of $\mathscr{M}^\ast $ as shown in \eqref{x3r109x}. One notices that the last two conditions in \eqref{werwgwre} requires that
\begin{align}\label{x340324cvv} &L \ge  \f{\X^{1/6+\varepsilon}}{q^{4/3}} \quad  \text{and}\quad
 L \ge \f{   \X ^{1/10}      }{ q^{6/5} M^{1/5-\varepsilon} },
\end{align} respectively, upon observing that $L^2\le \ell \le L^4$, whenever $\ell=\ell_1\ell_2$ or $\ell^2_1\ell^2_2$, with $\ell_1,\ell_2 \in \mathscr{P}_L$.
\subsubsection{Oscillatory case} Next, in this part we will focus on the situation where the parameter $C$ is suitably small such that $C\ll \sqrt{\X\ell}/M^{1+\varepsilon}$ in \eqref{ooo9944}. Recall \eqref{c35yub4u444}. In this case, the phase of the weight $U_\ell(y)$ will be roughly $\sqrt{\X \ell}/CM\gg \X^\varepsilon$. Set $\mathcal{A}:=4\pi\sqrt{\X\ell} /CM$, and put $\mathscr{J}(y)=U(y)F_k^+(\mathcal{A}\sqrt{y})$ for any $y\in \R^+$. To begin with, let us take a glance at the integral 
\bee 
\label{x3cer3gth5t}\int_{\R^+} \mathscr{J}(y)e(A\sqrt{y} - 3(yx)^{1/3})\ud y.\ene Notice that $\mathscr{J}$ is compactly supported on $[1/2, 5/2]$, and satisfies that $\mathscr{J}^{(j)}\ll_{j,k} \X^{j\varepsilon}$ for any $j>0$. On finds that the integral in $y$ above is negligibly small, unless that $\mathcal{A} \asymp x^{1/3}$, and, in which case it is essentially given by
\[x^{-1/6} e(-4x/\mathcal{A}^2) \mathscr{J}{\lf( (2x^{1/3}/\mathcal{A})^6 \ri)},\]up to an acceptable error term. It is clear that the integral in \eqref{x3cer3gth5t} with the positive sign in the exponential factor is negligibly small, by repeated integration by parts sufficiently many times. Inserting this back into 
\eqref{c35yub4u444} gives us the equivalent form of $\Psi (C; \ell, M, q) $ in \eqref{ooo9944} as follows:
\bee \label{x34r134rhhj}\fone{(\X \ell)^{1/4}} \sum_{c\sim C} \sum_{\substack{\pm , m,r\\m \ge 1\\
r|c M}} \f{\overline{A_F(m,1)}}{\sqrt{mr^2} }  \, e\lf (\pm \f{m r^2\overline{\ell q^2}}{c M}\ri)    e{\lf (-\f{mr^2}{ 2\pi^2 cM \ell q^2}\ri)}\mathscr{J}{\lf ({\f{mr^2}{\mathcal{M}} }\ri)},\ene 
with $\mathcal{M}\asymp \X^{1/2+\varepsilon} \ell^{3/2} q^2$ for any $\varepsilon>0$. As hinted by the argument in \S3.1.1 above, it suffices to concentrate on the $r=1$ case which, however, is thought to be responsible for the dominated contribution for the display in \eqref{x34r134rhhj}. Whilst, for the degenerate case where $r$ is suitably large, our analysis works as well, and the bound that we obtain is dominated by the former (as one will expect). In other words, we are simplified into \[\fone{(\X \ell)^{1/4}} \sum_{c\sim C} \sum_{\substack{m \ge 1}} \f{\overline{A_F(m,1)}}{\sqrt{m} }  \, e\lf ( \f{m \overline{\ell q^2}}{c M}\ri)    e{\lf (-\f{m}{ 2\pi^2 cM \ell q^2}\ri)}\mathscr{J}{\lf ({\f{m}{\mathcal{M}} }\ri)}.\]Here, as before, it suffices to consider the $`+'$ case in \eqref{x34r134rhhj}, on account of the resemblance of the treatments of both situations, $`+'$ and $`-'$. Recall that one has the coprimality relation $(cM,\ell q)=1$. Now, if one attempts to apply the reciprocity law and then the $\GL(3)$-Vorono\u{\i}, we finally arrive at 
\bee \label{414c1xx}\fone{(\X \ell)^{1/4}}  \sum_{n^\prime |\ell q^2}\, \sum_{c\sim C} \sum_{n\ge 1} \f{A_F(n^\prime, n)}{\sqrt{n}}\, S( cM, \pm n; \ell q ^2/n^\prime) \, \widetilde{\mathscr{J}_{\pm, \ell}} {\lf (\f{n{n^\prime }^2 \mathcal{M}}{ \ell^3 q^6 };\f{c}{C}\ri)}, \ene
akin to \eqref{xe1433}, where the integral transform is given by
\[\widetilde{\mathscr{J}_{\pm, \ell}} {\lf (x;\nu \ri)}=\int_{\R^+} \mathscr{J}{\lf (y\ri)}\, e{\lf ( \f{ \rho \mathcal{M} y}{ \nu C M \ell q^2}  \pm 3(yx)^{1/3}\ri)}\ud y\]for some non-zero constant $\rho\in \R$. Observe that the phase of the integral $\widetilde{\mathscr{J}_{\pm, \ell}}$ is still roughly $\sqrt{\X \ell}/CM$, and it is negligibly small, unless that $\mathcal{M}/(CM\ell q^2)\asymp x^{1/3}$, by repeated integration by parts sufficiently many times. This, in turn, indicates that the sum over $n$ is such that \[n{n^\prime }^2 \asymp \mathscr{N}:= \mathcal{M}^2/(CM)^3\ll \X^{1+\varepsilon} \ell ^3 q^4/(CM)^3.\]  

To proceed, one sees that it is necessary that ${\X \ell^3 q^4}/{(CM)^3}\gg  1$ for any $1\le C\ll \sqrt{\mathcal{X}\ell }/ M^{1+\varepsilon}$. This implies that $\ell 
\ge  \sqrt{M}/q^{2-\varepsilon}$, which is already ensured by the second pre-condition of \eqref{3x22r431} and the condition we have assumed for the case of $\ell=1$ on Page 17. Recall that $(\ell,q)=1$. Now, returning to \eqref{414c1xx}, there are two options for $n^\prime$, that is, 
\begin{itemize}
\item [(1)]  $n^\prime =1$;
\item [(2)] $n^\prime|\ell, n^\prime>1$, with $\ell \ge 2$.
\end{itemize}
Observe that, here, $n^\prime =q\text{ or } q^2$ cannot occur, since of 
that $q^2>\X^{1+100\varepsilon}\ell^3q^4/M^3$ by the final choice of $L$ as presented in \eqref{c4r1t55} below, upon recalling that $\ell =1,\ell_1\ell_2$ or $ \ell^2_1\ell^2_2$ with $\ell_1,\ell_2\in \mathscr{P}_L$, the modulo set defined as in \eqref{nxr4300509454} before. We will be ready to utilize Lemma \ref{201910000011svffs111} to the sum in \eqref{414c1xx}. It suffices to merely consider Case (1); the contribution from Case (2) would be relatively negligible, upon recalling the remark below the proof of Lemma \ref{201910000011svffs111}. As of now, it is clear that 
\[ x^i \nu ^j
 \f{\partial ^i}{\partial x^i} 
  \f{\partial ^j} {\partial \nu^j}\widetilde{\mathscr{J}_{\pm, \ell}} {\lf (x;\nu \ri)} \ll (\sqrt{\X \ell}/CM)^{i+j}\]
for any $i,j>0$. We now apply \eqref{xwxwf3433}, with $Z=1, Z_1=Z_2=\sqrt{\X \ell}/CM$, finding that the expression in \eqref{414c1xx} is dominated by
\begin{align*}
&\ll \fone{(\X \ell)^{1/4-\varepsilon}} \lf \{   \sqrt{ \mathscr{N} \ell q^2 c}+\lf (  \f{\sqrt{\X \ell}}{CM}\ri)^{1/2}  (\ell q^2)^{3/4}c+ c\sqrt{\mathscr{N}}\ri\}\\
&\ll \f{\X^{1/4+\varepsilon} \ell ^{7/4} q^{3}}{CM^{3/2}}+\f{\sqrt{C}(\ell q^2)^{3/4+\varepsilon}}{\sqrt{M}} +\f{ \X^{1/4+\varepsilon} \ell^{5/4} q^2}{\sqrt{C} M^{3/2}}.
\end{align*}
This is controlled by
\[\ll \f{\X^{1/4+\varepsilon} \ell^{7/4} q^{3}}{M^{3/2}} +\f{\X^{1/4+\varepsilon} \ell q^{3/2}}{M}+\f{\X^{1/4+\varepsilon} \ell^{5/4} q^2}{M^{3/2}}\] for any $1\le C\ll  \sqrt{\X \ell}/M^{1+\varepsilon}$. In other words, we have established the following
\begin{lemma} For any $ C\ll  \sqrt{\mathcal{X}\ell}/ M^{1+\varepsilon}$ with $\ell $ being as in \eqref{x4cr31466}, we obtain
\bee\label{xw25y22} \Psi (C; \ell, M, q) \ll \f{\X^{1/4+\varepsilon} \ell^{7/4} q^{3}}{M^{3/2}} +\f{\X^{1/4+\varepsilon} \ell q^{3/2}}{M}+\f{\X^{1/4+\varepsilon} \ell^{5/4} q^2}{M^{3/2}}.
\ene
\end{lemma}
Altogether, the various bounds in \eqref{3c23t252}, \eqref{ooo9944}, \eqref{xw25y} and \eqref{xw25y22} above will result in:
\begin{prop} \label{992jdjd}For any $\ell =\{1,\ell_1\ell_2, \ell^2_1\ell^2_2\},\ell_1,\ell_2\in \mathscr{P}_L$, with $L$ satisfying that the constraints \eqref{3x22r43}, \eqref{3x22r431} and \eqref{x340324cvv}, we have
\begin{align} \label{xlcl4o0r4}&\sum_{f\in \mathcal{B}^\ast_k(M)}\omega^{-1}_f \lambda_f(\ell)L(1/2, F\otimes f)\ll \X^\varepsilon 
\left ( \f{1}{\sqrt{\ell}} +
\f{\X^{1/6} \ell q^{5/3} }{M}+q\sqrt{\f{\ell }{ M}}+
 \f{\ell^2\X^{1/4}\sqrt{q}}{M}
\right. \notag\\
&\phantom{=\;\;}\left.   \hskip 2.1cm     +\f{\X^{1/4} \ell^{7/4} q^{3}}{M^{3/2}} +\f{\X^{1/4} \ell q^{3/2}}{M}+\f{\X^{1/4} \ell^{5/4} q^2}{M^{3/2}}+\f{\sqrt{\X}}{M^{3/2}}  +\f{q}{M} \right ).
\end{align}
\end{prop}
This shows \eqref{00vo9493x4t422} immediately after a little simplification, as required.
\subsection{Proof of  Corollary \ref{corodoe923}}In this section, we shall turn to the proof of Corollary \ref{corodoe923}. Recalling the definition of the amplifier $\mathfrak{A}_{f}$ as presented in \eqref{3c5vt24yv}, which is dependent on $f_0 \in \mathcal{B}^\ast_k(M),$ and the property \eqref{x414rx310993}, via the positivity of the central value $L(1/2, F\otimes f)$ for any $f\in \mathcal{B}^\ast_k(M)$, we obtain
\[\mathfrak{A}_{f_0}  L(1/2, F\otimes f_0)\ll \sum_{f\in \mathcal{B}^\ast_k(M)}\omega^{-1}_f \mathfrak{A}_{f}  L(1/2, F\otimes f).\] Now, as a direct consequence of Proposition \ref{992jdjd}, by summing over $\ell= 1,\ell_1\ell_2, \ell^2_1\ell^2_2$ with $\ell_1,\ell_2\in \mathscr{P}_L$, it turns out that
\begin{align}\label{c4v3y3koplk} &L(1/2, F\otimes f) \ll \X^\varepsilon 
\left ( \f{M}{L} +
{ L^4  q^{5/3} \X^{1/6}}+qL^2 \sqrt{{ M }}+
 {L^8\X^{1/4}}\sqrt{q} \notag
\right. \notag\\
&\phantom{=\;\;}\left.   \hskip4cm     + \f{\X^{1/4} L^{7} q^{3}}{\sqrt{M}} +{\X^{1/4} L^4q^{3/2}}+\f{\X^{1/4} L^{5} q^2}{\sqrt{M}}+\f{\sqrt{\X}}{L} +\sqrt{\f{\X}{M}} +q \right ).
\end{align}

To beat the convexity barrier, the parameter $L$ on the first line on the right-hand side of \eqref{c4v3y3koplk} satisfies that
\bee \label{c7}  \begin{split}
& L\ge \f{M^{1+\varepsilon}}{\sqrt{\X}}, \quad   L\le \f{\X^{1/12-\varepsilon}}{q^{5/12}}\\
&L\le \fone{\sqrt{q}} \lf ( \f{\X}{M}\ri)^{1/4-\varepsilon},  \quad  L\le \lf (\f{\X}{q^2}\ri )^{1/32-\varepsilon};
\end{split}\ene 
whilst, the second line requires that 
\bee \label{c72}  \begin{split}
&   L\le \f{\X^{1/28-\varepsilon} M^{1/14}}{q^{3/7}},\quad  L\le \lf (\f{\X}{q^6}\ri )^{1/16-\varepsilon} \quad L\le \f{\X^{1/20-\varepsilon} M^{1/10}}{ q ^{2/5}}.
\end{split}\ene 
Recall that $(M^{3/2} q^2)^{1-\delta} \le \X\le (M^{3/2} q^2)^{1+\varepsilon}$ for any small constants $\delta, \varepsilon>0$. Collecting all the inequalities \eqref{3x22r43}, \eqref{3x22r431}, \eqref{x340324cvv}, \eqref{c7} and \eqref{c72}, we infer that the effective range for $L$ should be 
\begin{align}\label{c4r1t55}\f{q^{1+\varepsilon}} {{M^{1/4}} }+\f{M^{(1+3\delta)/4+\varepsilon}}{q^{1-\delta}} +1\le L\le \min {\lf \{ M^{3/64-\varepsilon}, \f{M^{3/32-\varepsilon}}{{q}^{1/4}},   \,  \f{{M}^{1/8}}{q^{5/14+\varepsilon}}, \, \f{M^{7/40-\varepsilon}}{q^{3/10}}\ri\}}.\end{align}
This yields the range $ M^{13/64+\varepsilon }\le  q\le M^{11/40-\varepsilon}$, and thus gives the requirement as shown in \eqref{c9vv0343}. The assertion of Corollary \ref{corodoe923} now follows by combining with \eqref{c4v3y3koplk}.
\section{Proofs of Theorems \ref{32432515777} \& \ref{c45254544d}}
In this section, we shall be dedicated to the proofs of Theorem \ref{32432515777} and Theorem \ref{c45254544d}. 
\subsection{Proof of Theorem \ref{32432515777}}We shall consider the average of the first moments of $\rm \text{GL}(3)\times GL(2)$ and $\text{GL}(2)$ $L$-functions in a family with a `harmonic’ average of $k
\sim K$. To be precise, let $W$ be a smooth function compactly supported on $[1/2,5/2]$ with bounded derivatives. Let $f$ be a $\text{GL}(2)$-newform of weight $k\in 2\Z$ and full level. Let $K$ be a sufficiently large parameter. Let $\ell\ge 1 $ be an integer such that $K^{1-\delta}\ll \ell \ll K^{1+\varepsilon}$ for any $\delta,\varepsilon>0$. We attempt to obtain an upper-bound for 
\bee \label{rvrg443445}\sum_{k \equiv 0 \bmod 2}W{\lf ( \f{k-1}{K}\ri)}  \sum_{f\in \mathcal{B}^\ast_k}\omega^{-1}_f \lambda_f(\ell) L(1/2, F\otimes f).\ene
To begin with, we will give a precise expression for the central value of
$L(s, F\otimes f) $. This can be derived in a standard way from \cite[Theorem 5.3]{IK}. Indeed, analogous to \cite[Eqn. (4.11)]{HMQ}, one finds that
\begin{align} \label{dji3234}
L(1/2, F\otimes f) =(1+\varepsilon(F\otimes f))\sum_{m,n\ge 1}\f{\mu(m)A_F(n,1)\lambda_f(mn)}{\sqrt{m^3n}}V^\star\lf (\f{m^3n}{k^3}\ri) ,
\end{align}
where \[V^\star(y)=\fone{2 \pi i}\int_{(3)}  \lf (\cos\lf ( \f{\pi s}{4 A}\ri)\ri)^{-24A} \f{L_\infty (s+1/2, F\otimes f) }{L_\infty (1/2, F\otimes f)}    L(1+2s, F) y^{-s} \f{\ud s}{s},\]
with $A>0$ sufficiently large.
Upon shifting the line to $\text{Re}(s)=-1/2$, we get
\bee \label{4320xkkxe} V^\star (y) =L(1, F)+O(k^\varepsilon y^{1/2+\varepsilon}) \ene
for $y\le 1$. Whilst, for $y>1$, one might shift the line to $\text{Re}(s)=A$, obtaining \[ y^j {V^\star}^{(j)}(y) \ll \lf (\f{k^3}{y}\ri)^{A}\]for any $j\ge 0$, upon exploiting Stirling's estimates to \eqref{3rqr34 r}. As described in \S2.3, the root number $\varepsilon(F\otimes f)$ is dependent on the weight $k$; thus, by \eqref{dji3234}, we need to turn to seeking an effective control over the following two sums
\bee \label{rf45t5650}
\sum_{k \equiv 0 \bmod 2}W{\lf ( \f{k-1}{K}\ri)}  \sum_{f\in \mathcal{B}^\ast_k}\omega^{-1}_f  \lambda_f(\ell)\,\sum_{n\ge 1} \f{A_F(n,1)\lambda_f(n)}{\sqrt{n}} V^\star {\lf (\f{n}{\mathcal{Y}}\ri)},
\ene
and
\bee \label{rf45t5651}
\sum_{k \equiv 0 \bmod 2} i^k\, W{\lf ( \f{k-1}{K}\ri)}   \sum_{f\in \mathcal{B}^\ast_k}\omega^{-1}_f \lambda_f(\ell)\, \sum_{n\ge 1} \f{A_F(n,1)\lambda_f(n)}{\sqrt{n}} V^\star {\lf (\f{n}{\mathcal{Y}}\ri)},
\ene
where $k^{3-\delta}\le \mathcal{Y}\le k^{3-\varepsilon}$ for arbitrarily small constants $\delta, \varepsilon>0$. By equipping with Petersoon formula, Lemma \ref{20191001}, we split the expressions in \eqref{rf45t5651} and \eqref{rf45t5651} into the diagonal terms and the off-diagonal terms. Via \eqref{4320xkkxe}, both diagonal terms add up to 
 \begin{align} \label{rf4VF5t5651}
\sum_{k \equiv 0 \bmod 2} (1+i^k)\, W{\lf ( \f{k-1}{K}\ri)}   \f{A_F(\ell,1)\lambda_f(\ell)}{\sqrt{\ell}} V^\star {\lf (\f{\ell}{\mathcal{Y}}\ri)}=&\f{L(1, F)A_F(\ell,1)\lambda_f(\ell)K}{4 \sqrt{\ell}} \widehat{W}(0)\notag \\
&+O(K^{-1/2+\varepsilon} ),
\end{align}with $\widehat{W}$ being the Fourier transform of $W$, which is given by $\widehat{W}(x)=\int_{\R} W(y)e(xy)\ud y$ for any $x\in \R^+$.

 We shall merely handle the off-diagonal term for the second sum \eqref{rf45t5651} in the following. The same analysis works for the first sum \eqref{rf45t5650}. The main difference is that, for the first sum, the Lemma \ref{201910000t666011svffs111} will be used, which, however, produces a possibly weaker bound since of more error terms thereof; for the second one, we will apply Lemma \ref{20191000s111} instead. It is verifiable that the off-diagonal term for \eqref{rf45t5651} is presented as the following form
\[
2 \pi \sum_{c\ge 1}\sum_{k \equiv 0 \bmod 2} W{\lf ( \f{k-1}{K}\ri)}  \sum_{n\ge 1} \f{A_F(n,1)}{\sqrt{n}} \f{S(n,\ell;c)}{c} J_{k-1}{\lf( \f{4 \pi \sqrt{n\ell}}{c}\ri)}V^\star{\lf (\f{n}{\mathcal{Y}}\ri)}.
\]
Now, one applies Lemma \ref{20191000s111} as mentioned above, which yields two terms 
\bee \label{x3r3209c323114}
\mathcal{T}(K,\ell)=\sum_{c\ll \sqrt{\mathcal{Y}\ell}/K^{2-\varepsilon}}
\, \sum_{n\ge 1} \f{A_F(n,1)}{\sqrt{n}}  \f{S(n,\ell;c)}{c} \, e{\lf( \f{2 \sqrt{n\ell}}{c}\ri)}V^\star  {\lf (\f{n}{\mathcal{Y}}\ri)},
\ene  and
\bee \label{x3r3209c3234}
\widehat{\mathcal{T}}(K,\ell)=\sum_{\sqrt{\mathcal{Y}\ell}/K^{1-\varepsilon}\ll c \ll   \sqrt{\mathcal{Y}\ell}/K^{1-100\varepsilon}}\, \sum_{n\ge 1} \, \f{A_F(n,1)}{\sqrt{n}} \f{S(n,\ell;c)}{c} W{\lf (  \f{4\pi\sqrt{n\ell}}{cK}\ri)}  V^\star {\lf (\f{n}{\mathcal{Y}}\ri)}
\ene
plus two error terms $\text{Err}^\flat(K,\ell))$ and $\text{Err}^\natural( K,\ell)$, say, which can be estimated as
\bee \label{3rc3rc3}\begin{split}
\text{Err}^\flat(K,\ell))& \ll \f{1}{K^{2-\varepsilon}}\sum_{n\sim \Y} \f{|A_F(n,1)|}{\sqrt{n}} \sum_{c\sim  \sqrt{\mathcal{Y}\ell } /K} \f{|S(n,\ell;c)|}{c}\ll \f{\Y^{3/4}\ell^{1/4}}{K^{5/2-\varepsilon}},
\end{split}\ene
and \bee \label{3rc3rc44}\begin{split}
\text{Err}^\natural (K,\ell))&\ll \f{\sqrt{\ell}}{K^{4-\varepsilon}}\sum_{n\sim \Y} {|A_F(n,1)|} \sum_{c\ll  \sqrt{\ell } \mathcal{Y}^{1/2+\varepsilon} } \f{|S(n,\ell;c)|}{c^2}\ll  \f{\mathcal{Y}\sqrt{\ell}}{K^{4-\varepsilon}},
\end{split}\ene
respectively. Here, $\text{Err}^\flat(K,\ell)$ comes from the first error term in \eqref{xwx0987wf3433}; while, $\text{Err}^\natural (K,\ell)$ is for the second one. We have applied bound \eqref{9909909800099} in estimating $\text{Err}^\natural (K,\ell)$, when $k$ is sufficiently large.

In the following paragraphs, we will be devoted to the estimates of $\mathcal{T}$ and $\widehat{\mathcal{T}}$ given as in \eqref{x3r3209c323114} and \eqref{x3r3209c3234}, respectively. We first take care of $\mathcal{T}$. Observe that the $c$-sum in \eqref{x3r3209c323114} can be truncated at $ c\ll \Y^\varepsilon$ for any $\varepsilon>0$\footnote[4]{It should be pointed out that the $c$-sum is not negligible here, upon recalling the support of the function $\widecheck{W}$ defined in a fashion as in \eqref{x24ro226787}; at least, the  special case of $c=1$ stills makes an indispensable contribution, where it is necessary that $\ell \gg K^{1-\varepsilon}$.}. We attempt to apply Lemma \ref{0l5y43} with $\alpha=2 \sqrt{\ell}/c$, which shows the estimate that
  \bee \label{kjci34c42}{\mathcal{T}} (K,\ell)\ll \f{1}{\Y^{1/2-\varepsilon} } \lf (1+(\Y \ell)^{1/2+\varepsilon}\ri)\ll  \ell^{1/2+\varepsilon} .\ene
Next, let us take a look at the term $\widehat{\mathcal{T}}$ in \eqref{x3r3209c3234}. By invoking the Vorono\u{\i} summation formula, Lemma \ref{2017303534623421}, this sum is thus essentially transformed into
\begin{align*} \begin{split} \sup_{\sqrt{\mathcal{Y}\ell}/K^{1-\varepsilon}\ll C \ll   \sqrt{\mathcal{Y}\ell}/K^{1-100\varepsilon}} \f{C }{\sqrt{\mathcal{Y}}}& \sum_{c\sim C}\sum_{\substack{\pm , m^\ast,r\\m^\ast \ge 1\\
r|c}}  \f{\overline{A_{F}(m^\ast,r)}}{m^\ast r^2 }  \, e\lf (\mp \f{m^\ast r^2 \overline{\ell }}{c}\ri)\, \Omega_{\pm}\lf ({\f{m^\ast r^2 \mathcal{Y}}{c^3};{V}_{c,\ell}}\ri),
  \end{split} \end{align*} with $V_{c,\ell}(x)=V^\star(x) W{({4\pi \sqrt{x\Y\ell}}/{(cK)}
) } $ for any $x\in \R^+$. 
By Lemma \ref{3453464lljyrdfv432}, the display above turns out to be controlled by
\bee \label{cdewrwc2}\begin{split}  \sup_{\sqrt{\mathcal{Y}\ell}/K^{1-\varepsilon}\ll C \ll   \sqrt{\mathcal{Y}\ell}/K^{1-100\varepsilon}}  \f{\Y^{1/6}}{C }
& \sum_{c\sim C}\sum_{\substack{\pm , m^\ast,r\\m^\ast \ge 1\\
r|c}}    \f{\overline{A_{F}(m^\ast,r)}}{{m^\ast}^{1/3} r^{2/3} }  \, e\lf (\mp \f{m^\ast r^2\overline{\ell }}{c }\ri)\, \widehat{V_{\pm,\ell}}   \lf ({\f{m^\ast r^2 \mathcal{Y}}{c ^3}}\ri)
  ,\end{split}\ene 
 with the integral transform $\widehat{V_{\pm,\ell}} $ being
 \[ \widehat{V_{\pm,\ell}} (x)=\int_{\R^+} 
V^\star(y) W{\lf (\f{ 4\pi \sqrt{y\Y\ell}}{cK}
\ri ) }    \, e{\lf (
\pm 3(xy)^{1/3}\ri ) } \ud y.\]
A direct computation with the help of the Kim–Sarnak bound shows an upper-bound of $O(\Y^{1/6+\varepsilon})$ for \eqref{cdewrwc2}; this, however, is sufficient for our purpose. In summary, upon collecting the estimates in \eqref{rf4VF5t5651}, \eqref{3rc3rc3}, \eqref{3rc3rc44} and \eqref{kjci34c42}, we obtain 
 \begin{align}\label{0505404v}&\sum_{k \equiv 0 \bmod 2}W{\lf ( \f{k-1}{K}\ri)}   \sum_{f\in \mathcal{B}^\ast_k}\omega^{-1}_f
 \lambda_f(\ell) L(1/2, F\otimes f) \notag\\
 &\hskip 2.3cm \ll \Y^{\varepsilon }\lf (\f{K^{1+\varepsilon}}{\sqrt{\ell}}+\f{  \Y^{3/4} \ell^{1/4}  }{ K^{5/2-\varepsilon}}  + \f{\mathcal{Y}\sqrt{\ell}}{K^{4-\varepsilon}}+\ell^{1/2+\varepsilon} +\Y^{1/6+\varepsilon}\ri).
 \end{align}
Now, after divided by a factor $\sqrt{\ell}$ on both sides of \eqref{0505404v}, we sum over $K^{1-\delta}\ll \ell \ll K^{1+\varepsilon}$ for any $\delta,\varepsilon>0$, obtaining
  \begin{align*}& \sum_{k \equiv 0 \bmod 2}W{\lf ( \f{k-1}{K}\ri)}   \sum_{f\in \mathcal{B}^\ast_k}\omega^{-1}_fL(1/2,f)L(1/2, F\otimes f)\ll K^{1+\varepsilon}.
 \end{align*}This gives \eqref{009koieu3} immediately. 
  \subsection{Proof of Theorem \ref{c45254544d}} At the end of this paper, it remains to complete the proof of Theorem \ref{c45254544d}. This can be realized by a similar argument in \S4.1. Actually, by repeating the procedure in evaluating the multiple sum \eqref{rvrg443445} with $
  \ell=1$ as in \S4.1, one sees that the sum in \eqref{009koieu3ss} is boiled down to evaluating $\widehat{\mathcal{T}}(K,1)$, up to the error terms $\text{Err}^\flat(K,1)$ and $\text{Err}^\natural (K,1)$. We are thus allowed to deduce 
 \begin{align*}\sum_{k \equiv 0 \bmod 2}W{\lf ( \f{k-1}{K}\ri)}   \sum_{f\in \mathcal{B}^\ast_k}\omega^{-1}_f
L(1/2, F\otimes f) =\f{L(1, F)K}{4}\widehat{W}(0) &+\widehat{\mathcal{T}}(K,1)+\text{Err}^\flat(K,1)\\
&+\text{Err}^\natural (K,1)+O(K^{-1/2+\varepsilon}).
 \end{align*}  
Observe that \eqref{cdewrwc2} is estimated as $O(\Y^{-100})$ whenever $\ell=1$, which, however, implies that  $\widehat{\mathcal{T}}(K,1)$ is negligibly small in term of the magnitude. One concludes that 
 \begin{align*}\sum_{k \equiv 0 \bmod 2}W{\lf ( \f{k-1}{K}\ri)}   \sum_{f\in \mathcal{B}^\ast_k}\omega^{-1}_f
L(1/2, F\otimes f) 
&=\f{L(1, F)K}{4} \widehat{W}(0)+O(K^{-1/4+\varepsilon}), 
 \end{align*}  
as desired.

\end{document}